\documentclass{amsart}

\usepackage{color}
\usepackage{amsthm,amssymb,verbatim}
\usepackage{graphicx}
\usepackage{enumerate}

\newcommand{\D}{{\rm D}}

\newcommand{\trm}{\operatorname{Trace}_m}

\newcommand{\Rr}{\operatorname{R}}

\newcommand{\ee}{\mathbf{e}}
\newcommand{\vv}{\mathbf{v}}

 \newcommand{\Div}{\operatorname{div}}
 
 \newcommand{\HH}{\mathcal{H}}
 
 \newcommand{\LL}{\mathcal{L}}
 \newcommand{\spt}{\operatorname{spt}}

 \newcommand{\RR}{\mathbf{R}}  
 \newcommand{\BB}{\mathbf{B}}  
 \newcommand{\dist}{\operatorname{dist}}
 
 \newcommand{\area}{\operatorname{area}}
 \newcommand{\eps}{\epsilon}
 \newcommand{\Tan}{\operatorname{Tan}}

\newcommand{\grad}{\nabla}

 \newcommand{\xx}{\mathbf{x}}
 \newcommand{\yy}{\mathbf{y}}

\input epsf
\def\begfig {
\begin{figure}
\small }
\def\endfig {
\normalsiue
\end{figure}
}

\swapnumbers 

    \newtheorem{theorem}    {Theorem}       [section]
    \newtheorem{lemma}      [theorem]       {Lemma}
    \newtheorem{corollary}  [theorem]     {Corollary}
    \newtheorem{proposition}       [theorem]       {Proposition}

    \newtheorem*{theorem*}{Theorem}

    \theoremstyle{definition}
    \newtheorem{definition}  [theorem] {Definition}

    \theoremstyle{definition}
    \newtheorem{remark}   [theorem]       {Remark}
    \newtheorem{a-counterexample} [theorem] {A counterexample}

\usepackage{hyperref}
\usepackage[alphabetic, msc-links, backrefs]{amsrefs}

\begin{document}

\renewcommand{\thesubsection}{\thetheorem}

\title[Controlling Area Blow-up]{Controlling Area Blow-up in 
Minimal\\ or Bounded Mean Curvature Varieties}
\author{Brian White}
\address{Department of Mathematics, Stanford University, Stanford, CA 94305}
\thanks{This research was supported by the National Science Foundation
  under grants~DMS-0406209, DMS~1105330, and DMS~1404282}
\email{bcwhite@stanford.edu}
\date{November 27, 2012. Revised June 24, 2015.}

\begin{abstract}
Consider a sequence of minimal varieties $M_i$ in a Riemannian manifold $N$ such that
the measures of the boundaries are uniformly bounded on compact sets.  Let $Z$ be the set
of points at which the areas of the $M_i$ blow-up.  We prove that $Z$ behaves in some
ways like a minimal variety without boundary.  In particular, it satisfies the same maximum and barrier principles
that a smooth minimal  submanifold satisfies.   
For suitable open subsets $W$ of $N$, this allows one to show that if the areas of the $M_i$ are 
uniformly bounded on compact subsets of $W$, then the areas are in fact uniformly bounded on
all compact subsets of $N$.   Similar results are proved for varieties with bounded mean curvature.
The results about area blow-up sets are used to show that
the Allard Regularity Theorems can be applied in some 
situations where key hypotheses appear to be missing. 
In particular,
we prove a version of the Allard Boundary Regularity Theorem that does not require
any area bounds.  For example, we prove that if a sequence of smooth minimal submanifolds converge
as sets to a subset of a smooth, connected, properly embedded manifold 
with nonempty boundary, and if the convergence of the
boundaries is smooth, then the convergence is smooth everywhere.
\end{abstract}

\maketitle

\tableofcontents

\section{Introduction}
If $M$ is an embedded, $m$-dimensional manifold-with-boundary in a Riemannian manifold $\Omega$ and if $U$
is a subset of $\Omega$, then $|M|(U)$ and $|\partial M|(U)$ will denote
    the $m$-dimensional area of $M\cap U$ and the $(m-1)$-dimensional area of $(\partial M)\cap U$.
    For most of this article, the reader is not assumed to have any knowledge of varifolds, but
    for readers who do have such knowledge, if $M$
    an $m$-dimensional varifold, then $|M|(U)$ denotes
    the mass of $M$ in $U$ (written $\|M\|(U)$ in~\cite{allard-first-variation} 
    and $\mu_M(U)$ in~\cite{simon-book}),
    and $|\partial M|(U)$ denotes the generalized boundary measure of $M$ applied to
    the set $U$. (In~\cite{allard-first-variation}, $|\partial M|(U)$ is written $\|\delta M\|_{\rm sing}(U)$.)

Let $M_i$ be a sequence of $m$-dimensional minimal varieties in a Riemannian manifold $\Omega$,
or, more generally, varieties with mean curvature bounded by some $h<\infty$.  
Even more generally, the hypothesis that the mean curvature is bounded above by $h$ can be replaced
by the hypothesis that not ``too much'' of $M_i$ has mean curvature $>h$, i.e., that
\[
  \limsup_{i\to\infty} \int_{M_i\cap K} (|H|- h)^+\,dA < \infty
\]
for every compact subset $K$ of $\Omega$, where $H$ is the mean curvature vector and
where $t^+$ denotes the positive part of $t$ (that is, $t^+ = \max\{t, 0\}$).
 We suppose that 
the boundaries have uniformly bounded measure in compact sets $K$:
\[
    \limsup_{i\to\infty} |\partial M_i| (K)  < \infty.
\]
Let $Z$ be the set of points at which the areas of the $M_i$ blow up:
\[
  Z = \{ x\in \Omega: \text{ $\limsup_i |M_i|(\BB(x,r) = \infty$ for every $r>0$} \}
\]
Equivalently, $Z$ is the smallest closed subset of $\Omega$ such that the areas of the $M_i$
are uniformly bounded as $i\to\infty$ on compact subsets of $\Omega\setminus Z$.

It is useful to have natural conditions that imply that $Z$ is empty, since if $Z$ is empty, then
the areas of the $M_i$ are uniformly bounded on all compact subsets of $\Omega$
and thus (for example) a subsequence of the $M_i$ will converge as varifolds to a limit varifold
of locally bounded first variation.
This paper gives some such conditions.  It also gives some properties shared by every such area blowup set $Z$.

First we prove that every such set $Z$ satisfies the following maximum principle:

\begin{theorem}[Maximum Principle, \S\ref{main-theorem}]\label{intro-maximum-principle}
 If $f:\Omega\to \RR$ is a ${\rm C}^2$ function and if $f|Z$ has a local maximum at $p$,
then 
\[
    \trm(\D^2f(p)) \le h\, |\D f(p)|,
\]
 where $\trm(\D^2f(p))$ is the sum of the $m$ lowest eigenvalues
of the Hessian of $f$ at $p$.
\end{theorem}

A closed set $Z$ that satisfies the 
conclusion of Theorem~\ref{intro-maximum-principle} will be called an $(m,h)$ set.
The concept of an $(m,h)$ set can be regarded as a  generalization of the concept
of an $m$-dimensional, properly embedded submanifold without boundary and with mean curvature bounded
by $h$.  In particular, if $M$ is a smooth, properly embedded, $m$-dimensional submanifold without boundary,
then $M$ is an $(m,h)$ set if and only if its mean curvature is bounded by $h$.
(See Corollary~\ref{main-theorem-examples-corollary} and Theorem~\ref{barrier-theorem}.)

We also prove that any $(m,h)$ set $Z$ satisfies the same barrier principle that is satisfied by $m$-dimensional
submanifolds of mean curvature bounded by $h$:

\begin{theorem}[Barrier Principle, \S\ref{barrier-theorem}]\label{intro-barrier-theorem}
Let $\Omega$ be a ${\rm C}^1$ Riemannian manifold without boundary,
and let $Z$ be
an $(m,h)$ subset of $\Omega$.
Let $N$ be a closed region in $\Omega$ with smooth boundary such that 
    $Z\subset N$,
and let $p\in Z\cap \partial N$.  Then
\[
  \kappa_1 + \dots + \kappa_m \le h
\]
where $\kappa_1\le \kappa_2 \le \dots \le \kappa_{n-1}$ are the principal curvatures
of $\partial N$ at $p$ with respect to the unit normal that points into $N$.
\end{theorem}

The converse is also true.  (See Theorem~\ref{converse-theorem}.)

In the case $\dim(N)=m+1$, there is a also
a strong barrier principle (Theorem~\ref{strong-barrier-theorem}): in the 
notation of Theorem~\ref{intro-barrier-theorem}, if $\dim(N)=m+1$,
if the mean curvature of $\partial N$ is everywhere greater than or
equal to $h$ (with respect to the normal that points into $N$), and if $Z\subset N$ is an
$(m,h)$ set that touches $\partial N$ at a point $p$, then $Z$ contains the entire 
connected component of  $\partial N$ containing $p$.

The support of every $m$-dimensional varifold with mean curvature bounded by $h$ is 
an $(m,h)$ set. (See Corollary~\ref{main-theorem-examples-corollary}.)
Thus Theorem~\ref{intro-barrier-theorem}
includes as a special case the barrier principle for varifolds proved in~\cite{White-maximum}.

The following theorem allows one to conclude in some circumstances that $Z$ is empty:

\begin{theorem}[Constancy Theorem, \S\ref{constancy-theorem}]
Suppose that an $(m,h)$ set $Z$ is a subset of a connected, ${\rm C}^1$ properly embedded, $m$-dimensional 
submanifold $M$ of the ambient space $\Omega$. Then $Z=\emptyset$ or $Z=M$.  In other
words, the characteristic function of $Z$ is constant on $M$.
\end{theorem}

\begin{corollary}\label{intro-corollary}
Let $\Sigma$ be a closed, proper subset of a connected, ${\rm C}^1$-embedded, $m$-dimensional 
submanifold $M$ of $\Omega$ .  Suppose that $M_i$ is a sequence of $m$-dimensional varieties 
(or, more generally, varifolds)
in $\Omega$ such that the boundary measures of the $M_i$ are uniformly bounded on compact sets
and such that the mean curvatures of the $M_i$ are uniformly bounded.   If the areas of the $M_i$
are uniformly bounded on compact subsets of $\Omega\setminus \Sigma$, then they are also uniformly
bounded on compact subsets of $\Omega$.
\end{corollary}

In Section~\ref{allard-section}, 
we use the results above (specifically, Corollary~\ref{intro-corollary}) 
to prove a theorem (\S\ref{more-extended-allard-theorem}) that extends 
Allard's Regularity Theorem in the case of integer-multiplicity varifolds.
(Allard's Theorem holds more generally for varifolds with densities bounded below by $1$,
but our theorem is false under that weaker assumption: 
see~\S\ref{counterexample-section}.)
For example, for minimal varieties, we have:

\begin{theorem}[\S\ref{extended-allard-theorem}]\label{intro-allard-theorem}
Suppose $M_i$ is a sequence of proper $m$-dimensional minimal varieties-without-boundary
 (or stationary integral varifolds) 
in a Riemannian manifold $\Omega$.  Suppose the $M_i$ converge as sets (see
 remark~\ref{remark:kuratowski}) to a subset of an $m$-dimensional,
connected, ${\rm C}^1$ properly embedded submanifold $M$ of $\Omega$.  If the $M_i$ converge weakly to $M$ with
multiplicity one anywhere, then they converge smoothly to $M$ everywhere.
\end{theorem}

The key word here is ``anywhere": to invoke Allard's theorem directly, one needs to assume
the weak, multiplicity one convergence $M_i\to M$ everywhere.

An analogous result (\S\ref{more-extended-allard-theorem})  
 holds if the $M_i$ have uniformly bounded mean curvatures.  
 
Allard's Regularity Theorem does not require bounded mean curvature, but rather only
mean curvature in $\LL^p$ for some $p$ greater than the dimension.
Similarly, Theorem~\ref{more-extended-allard-theorem} does not require that the surfaces $M_i$
in question have bounded
mean curvature, but rather that they satisfy the weaker hypothesis that
\begin{equation}\label{eq:Lp-type-condition}
   \limsup_i\int_{M_i\cap K} ( |H|-h)^+)^p\,dA  < \infty
\end{equation}
for some $h<\infty$, for some $p>m$, and for every compact $K\subset \Omega$. In that case, the conclusion
is not smooth convergence but rather ${\rm C}^1$ convergence with local ${\rm C}^{1,1-m/p}$ bounds.

\begin{remark}\label{remark:kuratowski}
In Theorem~\ref{intro-allard-theorem} (and elsewhere in this paper),  we say that the sets $S_i\subset \Omega$
{\em converge as sets} to $S\subset\Omega$ provided:
\[
  S = \{ x\in \Omega: \limsup_i d(x, S_i) = 0 \}  = \{x\in \Omega: \liminf_i d(x,S_i) = 0\},
\]
where $d(x,A)=\inf\{ d(x,a): a\in A\}$.
This notion (due to Kuratowski) is in general weaker than convergence 
 with respect to the Hausdorff metric on closed sets. (For compact $\Omega$, the 
 two notions are equivalent.)   
Such convergence has a nice compactness property:
if $S_i$ is a sequence of closed subsets of a Riemannian manifold $\Omega$,
then (according to the Arzela-Ascoli theorem) 
by passing to a subsequence, we can assume that the functions $d(\cdot,S_i)$
converge to a limit function $f$.   It follows that the $S_i$ converge as sets to the zero
set of $f$.
\end{remark}

Section~\ref{allard-boundary-section} gives a version of Allard's Boundary Regularity Theorem 
that does not assume any area bounds.

Sections~\ref{minimal-hypersurfaces-section} and~\ref{bounded-mean-curvature-section}
 give additional results for the case of codimension one varieties, i.e., the case of $(m,h)$ sets
in an $(m+1)$-dimensional manifold.  For example, as a special case of those results, we have:

\begin{theorem}\label{intro-minimal-hypersurface-theorem}
Let $m<7$ and let $N$ be closed, mean convex region in $\RR^{m+1}$ with smooth boundary.
Suppose that $\partial N$ is not a minimal surface, and that $N$ does not contain
any smooth, stable, properly embedded minimal hypersurface.
If $Z$ is an $(m,0)$ set contained in $N$, then $Z=\emptyset$.
\end{theorem}

(We remark that in $\RR^3$, the hypothesis that $N$ not contain a smooth, stable
properly embedded minimal hypersurface is redundant.)

We also prove (Corollaries~\ref{halfspace-theorem-1} and~\ref{halfspace-theorem-2}) 
that the Hoffman-Meeks Halfspace Theorems for proper minimal surfaces in $\RR^3$
hold for arbitrary $(2,0)$ sets in $\RR^3$.

Finally, in section~\ref{distance-section}, we prove

\begin{theorem}
If $Z\subset \RR^n$ is an $(m,h)$ set, then the set $Z(s)$ of points at distance $\le s$ from $Z$
is also an $(m,h)$ set.
\end{theorem}

We also prove an analogous result for $(m,h)$ subsets of Riemannian manifolds.

Readers who are interested in minimal varieties (rather than varieties of bounded curvature)
may skip section~\ref{bounded-mean-curvature-section} 
and much of sections~\ref{allard-section} and~\ref{allard-boundary-section}.  
(The portions of~\ref{allard-section} 
and~\ref{allard-boundary-section} that may be
skipped are indicated there.) 

\section{Area Blowup}\label{area-blowup-section}

\begin{definition}\label{(m,h)-definition}
Let $\Omega$ be a smooth manifold without boundary and with a ${\rm C}^1$ Riemannian metric $g$.
Let $Z$ be a closed subset of $\Omega$.
We say that $Z$ is an {\em $(m,h)$ subset} of $(\Omega,g)$
provided it has the following property: if $f:\Omega\to \RR$ is a ${\rm C}^2$
function such that $f|Z$ has a local maximum at $p$, then
\begin{equation}\label{the-test}
  \trm(\D^2f(p)) \le  h\,|\D f(p)|. 
\end{equation}
    If there is such a function $f$ for which~\eqref{the-test} does not hold, we say that $Z$ fails
to be an $(m,h)$ set at the point $p$.
\end{definition}

Here $|\D f|$ is the norm of the derivative of $f$ (or, equivalently, the norm of the gradient of $f$)
with respect to the metric $g$, and 
 $\trm(\D^2f)$ is the sum of the lowest $m$ eigenvalues of the
 Hessian of $f$ with respect to the metric $g$.
 In other words, $\trm(\D^2f(p))$ is the sum of the lowest $m$ eigenvalues of the 
 matrix of second partial derivatives of $f$ in any system of normal coordinates at $p$.
Note this matrix is also the matrix (with respect to the same normal coordinates) for the endomorphism 
$\D (\nabla f(p))$, where $\D$ is the covariant derivative (with respect to the Levi-Civita
connection) and $\nabla f$ is the gradient of $f$.  Thus $\trm(\D^2f)$ is equal
to $\trm(\D \nabla f)$.

\begin{remark}\label{closed-set-of-h-remark}
Let $Z\subset \Omega$ be a closed set.  It follows immediately from the definition
that the set of $h$ for which $Z$ is an $(m,h)$ set either is empty or has the form $[\eta,\infty)$
for some $0\le \eta<\infty$.
\end{remark}

\begin{remark}
Suppose $N$ is a smooth Riemannian $n$-manifold with boundary and with a ${\rm C}^1$ Riemannian metric $g$.
Then $N$ can be embedded into a smooth open $n$-manifold $\Omega$ and the metric $g$ can be extended
to be a ${\rm C}^1$ Riemannian metric on all of $\Omega$.   A closed subset $Z$ of $N$ is called an $(m,h)$ subset
of $(N,g)$ if and only if it is an $(m,h)$ subset of $(\Omega,g)$.  It is straightforward to show that this condition is indepedent
of the choice of $\Omega$ and of choice of the extension of the metric.
\end{remark}

The following lemma implies that in the definition of $(m,h)$ subset, it suffices
to consider test functions $f$ with additional properties.

\begin{lemma}\label{special-f-lemma}
Suppose $Z\subset \Omega$ is a closed subset  that fails to be an $(m,h)$ subset at the point $p\in Z$.
Then there is a ${\rm C}^2$ function $f:\Omega\to \RR$ such that
\begin{enumerate}[\rm (1)]
\item\label{item:strictly-positive} $\trm(\D^2f(p)) > h\,|\D{}f(p)|$.
\item\label{item:unique} The restriction of $f$ to $Z$ attains its maximum value of $0$ uniquely at the point $p$:
\[  
     \text{$f(x) < f(p)=0$ for all $x\in Z$, $x\ne p$}.
\]
\item\label{item:proper}
The set $\{x: f(x)\ge a\}$ is compact for every $a\in \RR$.  Indeed, if $u:\Omega\to (-\infty, 0]$
is any smooth, proper function, we can choose $f$ so that $f$ coincides with $u$ outside of some compact
set.
\end{enumerate}
\end{lemma}

\begin{proof}
By hypothesis,  there is a ${\rm C}^2$ function
$f:\Omega\to \RR$ such that~\eqref{item:strictly-positive} holds and such that $f|Z$ has a local maximum
at $p$:
\[
     \max f | Z\cap\BB = f(p),
\]
where $\BB$ is some open neighborhood of $p$.
By replacing $f$ by $f-f(p)$, we can assume that $f(p)=0$.

Let $u:\Omega\to (-\infty, 0]$ be a smooth, proper function.
By modifying $u$ on a compact set, we can assume that $u(p)=0$, that $u(x)<0$ for all $x\ne p$,
and that $\D^2u(p)=0$.

Let $\phi:\Omega\to \RR$ be a smooth, nonnegative function that is supported in $\BB$ and that is equal to $1$ in some
neighborhood of $p$.  Replacing $f$ by $\phi f + u$ gives a function with all the asserted properties.
\end{proof}

The following corollary says that we can choose the function $f$ in Lemma~\ref{special-f-lemma} to be smooth
(not just ${\rm C}^2$), provided we are allowed to move the point $p$ slightly:

\begin{corollary}\label{special-f-corollary}
Suppose $Z\subset \Omega$ is a closed subset  that fails to be an $(m,h)$ subset at the point $q\in Z$.
Then there is a point $p\in Z$ (which may be chosen arbitrarily close to $q$) and a smooth
function $f:\Omega\to \RR$ having the 
properties asserted in Lemma~\ref{special-f-lemma}.
\end{corollary}

\begin{proof}
Let $f$ be a ${\rm C}^2$ function having all the properties asserted by Lemma~\ref{special-f-lemma}
 with $q$ in place of $p$.
Let $f_i:\Omega\to \RR$ be a sequence of smooth functions such that $f_i$ converges 
to $f$ uniformly and also locally in ${\rm C}^2$.   It follows that each $f_i|Z$ attains its maximum at some point $p_i$,
and that the $p_i$ converge to $q$.
Furthermore, the local ${\rm C}^2$ convergence implies that
\[
    \trm(\D^2f_i(p_i)) > h\,|\D{}f_i(p_i)|
\]
for all sufficiently large $i$.
For each such $i$, we can modify $f_i$ exactly as in the proof of Lemma~\ref{special-f-lemma}
to get a smooth function $\tilde f_i$ that has 
properties~\eqref{item:strictly-positive}, \eqref{item:unique}, and \eqref{item:proper} 
(with $\tilde f_i$ and $p_i$ in place of $f$ and $p$.)
\end{proof}

\begin{theorem}\label{main-theorem}
Let $\Omega$ be a smooth, $n$-dimensional manifold without boundary.
Let $g_i$ ($i=1,2,3,\dots$) and $g$ be ${\rm C}^1$ Riemannian metrics on $\Omega$
such that the $g_i$ converge to $g$ in ${\rm C}^1$.

For each $i$, let $M_i$ be an $m$-dimensional varifold in $\Omega$
such that the mean curvature of $M_i$ with respect to $g_i$ is bounded
by $h<\infty$ and such that the boundaries of the $M_i$ are uniformly bounded
on compact subsets of $\Omega$:
\begin{equation}\label{boundaries-controlled}
   \text{$\limsup_i |\partial M_i|(U) < \infty$ for all $U\subset\subset \Omega$}.
\end{equation}

Let 
\begin{equation}\label{Z-def}
   Z = \{ x\in \Omega: \text{$\limsup_i |M_i|(\BB(x,r)) = \infty$ for every $r>0$} \}.
\end{equation}

Then $Z$ is an $(m,h)$ subset of $\Omega$.

More generally, the conclusion remains true if the hypothesis that the mean curvatures
are bounded by $h$ is replaced by the hypothesis that
\begin{equation}\label{integral-bound}
  \limsup_i \int_{M_i\cap K} (|H|- h)^+\,dA <\infty
\end{equation}
for every compact $K\subset \Omega$, where $H$ is the mean curvature vector
and where $t^+ =\max\{t,0\}$.
\end{theorem}

\begin{remark}
Readers who are primarily interested in minimal varieties 
 may wish to read the following proof under that assumption that the $M_i$
are minimal, i.e., that $h= 0$: in that case a number of terms in the proof
drop out.   Similarly, readers primarily interested in bounded mean curvature varieties
may wish to make the assumption that $M_i$ has mean curvature bounded by $h$
(instead of making the more general assumption~\eqref{integral-bound}), since a few terms
in the proof then drop out.
\end{remark}

\begin{proof}  To simplify notation, we give the proof in case the $M_i$
are properly embedded manifolds-with-boundary.   But (aside from the notation)
exactly the same proof
works for general varifolds.  We prove the result by contradiction.
Thus suppose $Z$ fails to be an $(m,h)$ set at a point $p\in Z$.

By Lemma~\ref{special-f-lemma}, there is a ${\rm C}^2$ function $f:\Omega\to \RR$
such that 
\begin{gather}
    \trm(\D^2f(p)) >h\,  |\D f(p)| ,  \\
    \label{strict-minimum} \text{$f(x)< f(p)$ for every $x\in Z\setminus \{p\}$, and} \\
    \label{proper}\text{$\{f\ge a\}$ is compact for every $a\in \RR$}.
\end{gather}

Choose $\delta>0$ so that 
\[
  \left[     \trm(\D^2f)(p) - h\, |\D f(p)|  \, \right]_g > \delta.
\]
where the subscript $g$ indicates that the expression inside the brackets is
with respect to the metric $g$.
Let $\BB\subset \Omega$ be a compact ball centered at $p$
such that
\begin{equation}\label{min-trace}
 \min_\BB   \left[     \trm(\D^2f) - h\, |\D{}f| \,  \right]_g > \delta.
\end{equation}
(Such a set $\BB$ exists because the inequality $[\trm(\D^2f) - h\,|\D f|\,]_g > \delta$ 
defines an open set containing $p$.)

By~\eqref{strict-minimum} and~\eqref{proper},
\[
    \max_{Z \setminus\operatorname{interior}(\BB)} f  < f(p).
\]
By adding a constant to $f$, we can assume that
\begin{equation}\label{adjust-height}
    \max_{Z \setminus\operatorname{interior}(\BB)} f< 0 < f(p).
\end{equation}

Let $N=\{f\ge 0\}$.  By~\eqref{proper} and~\eqref{adjust-height}, 
 $N\setminus \operatorname{interior}(\BB)$ is a compact subset of $Z^c$,
so by definition of $Z$,
\begin{equation}\label{N-minus-B}
    \limsup_i |M_i| \, (N\setminus \BB) < \infty.
\end{equation}
Let $\BB^*$ be a small closed ball centered at $p$ such that $\BB^*$ is in the interior
of $\BB\cap N$.  
Choose constants $\Gamma$, $\gamma>0$, and $\tau\ge 0$ such that
\begin{align}
\label{Gamma-intro}
\max_N f\,|\D f|_g &< \Gamma,
\\
\label{gamma-intro}
\min_{\BB^*} f &> \gamma,\quad\text{and}
\\
\label{tau-intro}
\min_N [\,f\,\trm(\D^2f)]_g &> -\tau.
\end{align}
Note that the left sides of~\eqref{min-trace}, \eqref{Gamma-intro},
 and~\eqref{tau-intro}
all depend ${\rm C}^1$-continuously on the metric $g$.
Thus for all sufficiently large $i$, the inequalities hold with $g_i$
in place of $g$; for the rest of the proof, we restrict ourselves to such $i$.
All metric-dependent quantities below are with respect to $g_i$.

Let $X_i$ be the gradient of $\frac12 f^2$ with respect to the metric $g_i$:
\[
X_i=
\grad\left(\frac12 f^2\right)
=
f \grad f.
\]
Thus
\[
       |X_i| = f \,|\D f| \le \Gamma
\]
on $N$ by~\eqref{Gamma-intro}.
Also,
\[
\D X_i = f \D (\nabla f) + \D f \otimes \nabla f,
\]
so
\begin{equation}\label{eq:traces}
\trm(\D X_i) = \trm(f \D (\nabla f) + \D f \otimes \nabla f).
\end{equation}
The endomorphisms $f \D (\nabla f)$ and $\D f \otimes \nabla f$ of the tangent space (to the ambient space)
 correspond (using the metric) to the quadratic forms $f \D^2f$ and $df\otimes df$,
so~\eqref{eq:traces} implies that
\begin{align*}
\trm(\D X_i) 
&= \trm( f \D^2 f  + df\otimes df)
\\
&\ge \trm (f \D^2 f).
\end{align*}
This last inequality holds because $df\otimes df$ is positive semidefinite, which implies
that the eigenvalues of $f \D^2f+ df\otimes df$
are bounded below by the corresponding eigenvalues of $f\D^2f$.
(See Lemma~\ref{rayleigh-lemma}.)
Thus
\begin{equation*}
\trm(\D X_i)
\ge
\trm( f \D^2 f)
=
f \trm( \D^2 f)
\end{equation*}
on $N$.

Since $\Div_{M_i}X_i \ge \trm(\D X_i)$, this implies that
\begin{equation}\label{Div-bound}
\Div_{M_i}X_i 
\ge
\begin{cases} 
f\, (\,|\D f| \,h + \delta) &\text{on $N\cap\BB$} 
\\
 -\tau      &\text{on $N\setminus \BB$}
\end{cases}
\end{equation}
by~\eqref{min-trace} and~\eqref{tau-intro} (for $g_i$).

Since $X_i=0$ on $\partial N$, we have
\begin{align*}
\int_{M_i\cap N}\Div_{M_i}X_i\,dA 
&= \int_{M_i\cap N}-H\cdot X_i\,dA + \int_{\partial M_i\cap N} X_i\cdot \nu\,dS
\\
&\le \int_{M_i\cap N} h \,|X_i|\,dA 
   +\int_{M_i\cap N}(|H|-h)^+\,|X_i|\,dA + \int_{\partial M_i\cap N} |X_i|\,dS
\\
&\le \int_{M_i\cap N} h f \,|\D f|\,dA + 
    \Gamma\, \int_{M_i\cap N} (|H|-h)^+\,dA +\Gamma\, |\partial M_i|(N)
\\
&\le \int_{M_i\cap N} h f\, |\D f|\,dA + O(1)
\end{align*}
where $O(1)$ stands for any quantity that is bounded independent of $i$.
Thus
\[
\int_{M_i\cap N\cap \BB} (\Div_{M_i}X_i - h f \,|\D f|\,)\,dA 
\le 
\int_{M_i\cap (N\setminus \BB)} (h f\, \,|\D f|  - \Div_{M_i}X_i)\,dA + O(1).
\]
Thus by~\eqref{Div-bound},
\begin{align*}
\int_{M_i\cap N\cap \BB} \delta f\,dA 
&\le
\int_{M_i\cap(N\setminus \BB)} (h f \,|\D f| + \tau)\,dA + O(1)
\\
&\le
(\Gamma h + \tau)\, |M_i|\,(N\setminus \BB) + O(1)
\\
&\le
O(1),
\end{align*}
where the last step is by~\eqref{N-minus-B}.
Since $\BB^*\subset N\cap \BB$ and since $f>\gamma$ on $\BB^*$, this implies that
\begin{equation}\label{final}
\delta \gamma\, |M_i|\, ( \BB^*) 
\le
 O(1).
\end{equation}
However, the left side of~\eqref{final} is unbounded 
since $p\in Z$ and $\BB^*$ is a ball centered at $p$. The contradiction proves the theorem.  
\end{proof}

\begin{corollary}\label{main-theorem-examples-corollary}
Let $M$ be a proper, $m$-dimensional submanifold of $\Omega$ with no boundary
and with mean curvature everywhere $\le h$. 
Then $M$ is an $(m,h)$ subset of $\Omega$.

 More generally, let $M$ be an
 $m$-dimensional varifold (not necessarily rectifiable) 
 of locally bounded first variation with mean curvature everywhere $\le h$
and with no generalized boundary.  Then the support of $M$ is an $(m,h)$ subset of $\Omega$.
\end{corollary}

\begin{proof}
If $M$ is a manifold, let $M_i$ (for $i=1,2,\dots$) be obtained by multiplying the multiplicity of $M$ everywhere by $i$.
Then the area blowup set $Z$ is $M$ itself, and so $M=Z$ is an $(m,h)$
set  by Theorem~\ref{main-theorem}. 
  Similarly, if $M$ is a general $m$-varifold (i.e., a measure on a certain Grassman bundle),
one lets $M_i$ be the result of multiplying $M$ by $i$.  Exactly the same argument shows
that the support of $M$ is an $(m,h)$ set.
\end{proof}

Conversely, if a smooth $m$-dimensional manifold is an $(m,h)$ set, then its mean curvature is everywhere $\le h$.
(This follows from the Barrier Principle~\ref{barrier-theorem}.)

\section{Limits of $(m,h)$ subsets}\label{limits-section}

\begin{theorem}\label{limit-theorem}
Suppose for $i=1,2,3,\dots$ that $Z_i$ is an $(m,h_i)$ subset
of ${\rm C}^1$ Riemannian manifold $(\Omega, g_i)$.
Suppose that the $g_i$ converge in ${\rm C}^1$ to a Riemannian metric $g$,
that the $Z_i$ converge as sets (see remark~\ref{remark:kuratowski}) 
 to a closed set $Z$,
and that the $h_i$ converge to a limit $h$.

Then $Z$ is an $(m,h)$ subset of $(\Omega,g)$.
\end{theorem}

\begin{proof}
We prove the theorem by contradiction.
Suppose that $Z$ fails to be an $(m,h)$ subset at some point $p\in Z$.
By Lemma~\ref{special-f-lemma}, there is a ${\rm C}^2$ function $f:\Omega\to \RR$
such that 
\begin{gather}
    \label{limits:strict-inequality}  \trm(\D^2f(p)) > h\, |\D f(p)|  \\
    \label{limits:strict-minimum} \text{$f(x)< f(p)$ for every $x\in Z\setminus \{p\}$, and} \\
    \label{limits:proper}\text{$\{f\ge a\}$ is compact for every $a\in \RR$}.
\end{gather}

Now $Z_i$ is nonempty for all sufficiently large $i$, so by properness~\eqref{limits:proper},
$f|Z_i$ will attain its maximum at a point $p_i$.   
Furthermore, $p_i$ converges
to $p$ as $i\to\infty$ (by~\eqref{limits:strict-minimum} and~\eqref{limits:proper}).
 By~\eqref{limits:strict-inequality},
and by  the convergence of $g_i$ to $g$,
$h_i$ to $h$, and $p_i$ to $p$,
\[
  [\trm (\D^2f(p_i)) - h_i\,|\D f(p_i)|]_{g_i}>0,
\]
for all sufficiently large $i$, 
contradicting the hypothesis that $Z_i$ is an $(m,h_i)$ subset of $(\Omega,g_i)$.
\end{proof}

\begin{corollary}\label{blowup-corollary}
Suppose $Z$ is an $(m,h)$ subset of $(\Omega,g)$, where $\Omega$ is an open
subset of $\RR^n$  containing the origin, 
$g$ is a ${\rm C}^1$ Riemannian metric on $\Omega$, 
and $g_{ij}(0)$ is the Euclidean metric $\delta_{ij}$.
Let $\lambda_i$ be a sequence positive numbers tending to $\infty$, and suppose
that the dilated sets
\[
   \lambda_i Z := \{ \lambda_i p: p\in Z\}
\]
converge to a limit set $Z^*$.  Then $Z^*$ is an $(m,0)$ subset of $\RR^n$ (with respect to
the Euclidean metric.)
\end{corollary}

\begin{proof}
Let $g_i$ be the metric on $\lambda_i\Omega$ obtained from $g$ by dilation.
(In other words, $g_i$ is the result of pushing forward the metric $g$ by the map $x\mapsto \lambda_ix$,
and then multiplying by $\lambda_i^2$.)

Then $Z_i$ is an $(m, h/\lambda_i)$ subset of $(\Omega_i, g_i)$, so by 
Theorem~\ref{limit-theorem},
$Z$ is an $(m,0)$ subset of $\RR^n$.
\end{proof}

\section{The Constancy Theorem}\label{constancy-section}
   
\begin{theorem}[Constancy Theorem]\label{constancy-theorem}
Let $\Omega$ be an open subset of a manifold with ${\rm C}^1$ Riemannian metric $g$.
Let $Z$ be an $(m,h)$ set in $(\Omega,g)$.
Suppose $Z$ is a subset of a connected, $m$-dimensional, properly embedded submanifold $M$
of $\Omega$.
Then $Z=\emptyset$ or $Z=M$.  In other words, the characteristic function of $Z$ is constant on $M$.
\end{theorem}

\begin{proof}
The result is essentially local, so we may assume that $\Omega\subset \RR^n$.
Suppose the result is false, i.e., that $Z$ is a nonempty proper subset of $M$.
Then $M\setminus Z$ contains an open geodesic ball $B$ 
whose boundary contains a point $p\in Z$.  (See Lemma~\ref{closed-set-lemma} below if that is not clear.)
By translation, we can assume that $p=0$.  By making a linear change of coordinates, we may
assume that the metric $g$ is the Euclidean metric at $0$ (i.e., that $g_{ij}(0)=\delta_{ij}$.)

Now let $\lambda_i$ be a sequence of positive numbers such that $\lambda_i\to\infty$.
Note that the sets
\[
  \lambda_i(M\setminus B) := \{ \lambda_i x: x\in M\setminus B\}
\]
converge to a closed halfspace $H$ of $\Tan_0M$ with $0\in\partial H$.
Thus by passing to a subsequence, we can assume that the sets $\lambda_iZ$
converge to a closed subset $Z^*$ of $H$ with $0\in Z^*\cap \partial H$.
By rotating, we can assume that $H$ is the halfplane
\[
  H = \{x\in \RR^n: \text{$x_1\le 0$ and $x_i=0$ for all $i>m$}  \}.
\]
By Corollary~\ref{blowup-corollary}, $Z^*$ is an $(m,0)$ subset of $\RR^n$ (with respect to
the Euclidean metric.)

Now consider the function
\begin{align*}
&f: \RR^n\to\RR  \\
&f(x) = x_1 + (x_1)^2 + \sum_{i>m} (x_i)^2.
\end{align*}
Note that $f|H$ has a local maximum at $0$, so
  $f | Z^*$ has a local maximum at $0$.
But
\[
  \trm (\D^2f(0)) = 2 > 1 = |\D f(0)|,
\]
 contradicting the fact that $Z^*$ is an $(m,0)$ set.
\end{proof}

\begin{corollary}\label{constancy-corollary}
Suppose that $Z$ is an $(m,h)$ subset of $\Omega$.
Suppose also that $Z$ is contained in $M$, where $M$ is either
\begin{enumerate}
\item a ${\rm C}^1$ submanifold of dimension $\le m-1$, or
\item a connected, $m$-dimensional, ${\rm C}^1$ manifold-with-boundary such that the boundary is nonempty.
\end{enumerate}
Then $Z=\emptyset$.
\end{corollary}

\begin{proof}
Note that (in either case) $M$ is contained in an $m$-dimensional, ${\rm C}^1$ manifold $\hat{M}$ without boundary.
Now apply the Constancy Theorem~\ref{constancy-theorem} to $Z$ and $\hat{M}$.
\end{proof}

\begin{lemma}\label{closed-set-lemma}
Let $M$ be a connected Riemannian manifold without boundary.
Let $K$ be a proper, nonempty, closed subset of $M$.  Then $M\setminus K$
contains an open geodesic ball $B$ whose boundary contains
a point in $K$.
\end{lemma}

\begin{proof}
Let $q$ be a point in the boundary of $K$, i.e, in $K\cap\overline{M\setminus K}$.
Choose a point $p\in M\setminus K$ sufficiently close to $q$ that the
closed geodesic ball of radius $\dist(p,q)$ about $p$ is compact.
Then the open geodesic ball of radius $\dist(p,K)$ centered at $p$ has the desired
properties.
\end{proof}

\section{Versions of Allard's Regularity Theorem}\label{allard-section}

We begin with the case of minimal varieties:

\begin{theorem}\label{extended-allard-theorem}
Let $\Omega$ be a smooth Riemannian manifold (not necessarily complete).
Let $M_i\subset\Omega$ be a sequence of $m$-dimensional, properly embedded minimal submanifolds 
without boundary.   
Suppose that the $M_i$ converge as sets
(see remark~\ref{remark:kuratowski}) to a  
subset of an $m$-dimensional, connected, smoothly embedded submanifold $M$ 
of $\Omega$.
Suppose also that some point in $M$ has a neighborhood $U\subset \Omega$ such that $M_i\cap U$ 
converges weakly to $M\cap U$ with multiplicity one, i.e., such that
 \[
    \int_{M_i} \phi\,dA \to \int_{M}\phi\,dA \tag{*}
 \]
 for every continuous, compactly supported function $\phi: U\to\RR$.
Then $M_i$ converges to $M$ smoothly and with multiplicity one everywhere.

The result remains true if each $M_i$ is minimal with respect to a Riemannian metric $g_i$
provided the metrics $g_i$ converge smoothly to a limit Riemannian metric.
The result is also true if each $M_i$ is a $g_i$-stationary integral varifold, or, more generally,
a $g_i$-stationary varifold with density $\ge 1$ at every point in its support.
\end{theorem}

\begin{proof}
By Theorem~\ref{main-theorem}, the area blowup set $Z$ is an $(m,0)$ set.
By hypothesis, the area blowup set $Z$ is disjoint from $U$ and is therefore a proper subset
of $M$.  
Hence by the Constancy Theorem~\ref{constancy-theorem}, $Z=\emptyset$.  In other words, the areas of the $M_i$
are uniformly bounded on compact subsets of $\Omega$.
Thus (after passing to a subsequence)
the $M_i$ converge in the varifold sense to a stationary varifold $V$ supported in $M$.

By the constancy theorem for stationary varifolds 
 (\cite{allard-first-variation}*{\S 4.6(3)} or
     \cite{simon-book}*{\S41}),  $V$ is  $M$ with some constant
multiplicity.   By hypothesis, the multiplicity is equal to $1$ in $U$.  Therefore it is equal to $1$
everywhere.  But then the convergence $M_i\to M$ is smooth by the Allard Regularity Theorem. 
(More precisely, the convergence is $C^{1,\alpha}$ for some $\alpha>0$ by Allard's theorem,
which then implies by standard elliptic regularity that the convergence is smooth.)
\end{proof}

\begin{remark}
In case the $M_i$ are smooth minimal submanifolds, the proof actually requires very little geometric measure
theory.  In particular, the existence of a varifold limit and the constancy theorem follow rather directly from
the definition of varifold.  And if the $M_i$'s are smooth, the required version of the Allard Regularity Theorem
has a very elementary proof: see \cite{white-local-regularity}*{Theorem~1.1}.  (The proof
in~\cite{white-local-regularity} is for compact $M$, but with minor modification, 
the proof works for noncompact $M$.)
\end{remark}

\begin{remark}\label{multiplicity-remark}
The multiplicity one hypothesis~\thetag{*} here (and also in Theorem~\ref{more-extended-allard-theorem})
can be weakened to:
\[
   \limsup_i \int_{M_i}\phi\,dA \le \int_M \phi\,dA
\]
for every continuous, nonnegative, compactly supported function $f:U\to\RR$,
provided we assume that the $M_i$ converge as a sets to a {\em nonempty} subset
of $M$.  The proof is almost exactly as before.
\end{remark}

Readers interested in minimal (rather than bounded mean curvature) varieties may skip
the rest of this section.

\begin{theorem}\label{more-extended-allard-theorem}
Let $\Omega$ be a smooth Riemannian manifold (not necessarily complete).
Let $M_i\subset \Omega$ be a sequence of $m$-dimensional submanifolds without boundary
such that
\begin{equation}\label{eq:p-hypothesis}
   \limsup_i\int_{M_i\cap K} ( |H|-h)^+)^p\,dA  < \infty
\end{equation}
for some some $h<\infty$, some $p>m$, and for every compact $K\subset \Omega$.  
Suppose that the $M_i$ converge as sets
(see remark~\ref{remark:kuratowski}) to a  
subset of an $m$-dimensional, connected, ${\rm C}^1$ embedded submanifold $M$ 
of $\Omega$.
Suppose also that some point in $M$ has a neighborhood $U\subset \Omega$ such that $M_i\cap U$ 
converges weakly to $M\cap U$ with multiplicity one.
Then $M_i$ converges to $M$ in ${\rm C}^1$.  
Furthermore, the $M_i$ are locally uniformly bounded in $C^{1,1-m/p}$:
\[
  \limsup_i \left( \sup_{x, y\in M_i\cap K,\, x\ne y} \frac{d(\Tan(M_i,x), \Tan(M_i,y))}{d(x,y)^{1-m/p}} \right) < \infty
\]
for every compact $K\subset \Omega$.

The result remains true if the $M_i$ are integer-multiplicity rectifiable varifolds, or, more generally,
if the $M_i$ are varifolds with the gap $\alpha$ property (see Definition~\ref{gap-alpha-definition} below) 
for some $\alpha>1$.
\end{theorem}

The multiplicity one hypothesis in $U$ can be weakened slightly; see Remark~\ref{multiplicity-remark}. 

Unlike the minimal case (Theorem~\ref{extended-allard-theorem}), 
Theorem~\ref{more-extended-allard-theorem} fails for varifolds if the gap $\alpha$ hypothesis
is replaced by the weaker hypothesis that the density is $\ge 1$ almost everywhere.
See~\S\ref{counterexample-section} for an example of such failure.

\begin{definition}\label{gap-alpha-definition}
Let $V$ be an rectifiable $m$-varifold and $\alpha>1$.
We say that $V$ has the {\em gap $\alpha$ property} if
the density $\Theta(V,x)$ of $V$ at $x$ belongs to
\[
     \{1\} \cup [\alpha, \infty)
\]
for $\mu_V$ almost every $x$.
\end{definition}

\begin{proof}[Proof of Theorem~\ref{more-extended-allard-theorem}]
Let $h'=h+1$.
Then
\[
   (|H|-h')^+ \le ((|H|-h)^+)^p,
\]
so, by~\eqref{eq:p-hypothesis}, 
\[
   \limsup_i\int_{M_i\cap K} ( |H|-h')^+\,dA  < \infty
\]
for every compact $K\subset \Omega$.
Hence, exactly as in the proof of Theorem~\ref{extended-allard-theorem}, 
the areas of the $M_i$ must be uniformly bounded on compact sets.  Thus 
by passing to a subsequence, we can assume that the $M_i$'s converge to a limit varifold $V$.
Also, 
\begin{align*}
 \left( \int_{M_i\cap K} |H|^p\,dA \right)^{1/p}
&\le
\left( \int_{M_i\cap K} h^p\,dA \right)^{1/p}
+ 
\left( \int_{M_i\cap K} ((|H|-h)^+)^p\,dA \right)^{1/p}
\\
&\le
h \area(M_i\cap K)^{1/p}
+
 \left( \int_{M_i\cap K} ((|H|-h)^+)^p\,dA \right)^{1/p}
\end{align*}
Since the area of $M_i\cap K$ is bounded as $i\to\infty$,
the hypothesis~\eqref{eq:p-hypothesis} implies that
\[
 \limsup_i \left( \int_{M_i\cap K} |H|^p\,dA \right)^{1/p} < \infty.
\]

Recall that the density of $V$ at $x$ is
\[
  \Theta(V,x) = \lim_{r\to 0}\frac{|V|\BB(x,r)}{\omega_mr^m},
\]
provided the limit exists, where $\omega_m$ is the volume of the unit ball in $\RR^m$.
By the monotonicity formula for the $M_i$'s (which implies the same monotoncity for $V$),
$\Theta(V,x)$ exists everywhere and is upper semicontinuous in $x$, and it also has the 
property:
\[
   \text{$\Theta(V,x)\ge \limsup\Theta(M_i,x_i)$ provided $x_i\to x$}.
\]
(See, for example,~\cite{simon-book}*{\S17.8} for proof of these upper-semicontinuity
properties.)
In particular, $\Theta(V,x)\ge 1$ for every point in $\spt(V)$.

Now let $W$ be the set of points where $\Theta(V,x)=1$.  By hypothesis, $W$ is a nonempty
subset of $M$.  

We claim that $W$ is a relatively open subset of $M$.  To see that, suppose $x\in W$, i.e., that
$\Theta(V,x)=1$.  By the Allard Regularity Theorem, there is a open ball $B\subset \Omega$ around $x$
such that for all sufficiently large $i\ge i_0$, $\spt(M_i)\cap B$ is a ${\rm C}^1$ submanifold
and such that the $\spt(M_i)\cap U$ converge in ${\rm C}^1$ to $M\,\cap\, U$.   
By the upper-semicontinuity of density~\cite{simon-book}*{\S17.8},
and by replacing $B$ by a smaller open ball around $x$ and by replacing $i_0$ by a larger number,
we can assume that $\Theta(M_i,\cdot)<\alpha$ at almost all points of $\spt(M_i)\cap B$ for $i\ge i_0$.
By the gap $\alpha$ property, $\Theta(M_i,\cdot)=1$ at almost all points of $\spt(M_i)\cap B$ for $i\ge i_0$.
Hence the measures $\mu_{M_i}$ and $\HH^m\llcorner (\spt(M_i))$ coincide in $B$,
which implies (because of the ${\rm C}^1$ convergence) that $\mu_V$ and $\HH^m\llcorner M$ also coincide
in $B$. That in turn implies that $\Theta(V,\cdot)\equiv 1$ in $B\cap M$, so $B\cap M\subset W$.
This proves that $W$ is a relatively open subset of $M$.

Now if $W\ne M$, then  there would be an open geodesic ball $D$ in $W$ and a point $x\in \overline{D}\cap W^c$. (Recall Lemma~\ref{closed-set-lemma}.)
By definition of $W$, 
\[
   \Theta(V,x)\ne 1.
\]
Now the tangent cone $C$ to $V$ at $x$ is a plane
with multiplicity $\Theta(V,x)$. 
 (The multiplicity is constant on the plane by the
 constancy theorem for stationary varifolds~(\cite{allard-first-variation}*{\S4.6(3)} or
     \cite{simon-book}*{\S41}).
However, the multiplicity is equal to $1$ on a halfplane of that cone, namely the tangent halfplane to $D$ at $x$,
so $\Theta(V,x)=1$, contradicting the fact that $x\notin W$.
The contradiction proves that $W=M$, i.e, that $\Theta(V,\cdot)\equiv 1$
on $M$.  The conclusion then follows from the Allard Regularity Theorem.
\end{proof}

\begin{a-counterexample}\label{counterexample-section}
As mentioned above, Theorem~\ref{more-extended-allard-theorem} fails if the gap $\alpha$ hypothesis is replace by the hypothesis that
the density $\ge 1$ almost everywhere. We now give an example of that failure.
Let $g:\RR\to \RR$ be a smooth function such that $g(x)=0$ if and only if $|x|\ge 1$.
Let $S_n$ be the union of the graph of $(1/n)g$ and the $x$-axis (i.e, the graph of the $0$ function).
Let $\phi:[1,\infty)\to [1,2]$  be a smooth function such that $\phi(t)=2$ for $1\le t\le 2$ and such
that $\phi(x)=1$ for $x\ge 3$.

Now let $M_n$ be the rectifiable varifold whose support is $S_n$ and whose
density $\Theta(V,(x,y))$ at $(x,y)\in S_n$ is $1$ if $|x|<1$ and $\phi(|x|)$ for $|x|\ge 1$.
Let $M$ be the $x$-axis.  Then $M_n$, $M$, and $\Omega=\RR^2$ satisfy all 
the hypotheses except for the gap $\alpha$ hypothesis.  Also, $\Theta(M_n,\cdot)\ge 1$
at every point of $\spt(M_n)$, i.e., at every point of $S_n$.  However,
we do not have ${\rm C}^1$ convergence $\spt(M_i)\to M$.  Indeed, none of the $M_i$ are
${\rm C}^1$ at the points $(1,0)$ and $(-1,0)$.
\end{a-counterexample}

\section{Versions of Allard's Boundary Regularity Theorem}\label{allard-boundary-section}

\begin{theorem}\label{extended-allard-boundary-theorem}
Let $\Omega$ be a smooth Riemannian manifold
and let $M\subset \Omega$ be an $m$-dimensional smooth, connected, properly embedded 
manifold-with-boundary
such that $\partial M$ is smooth and nonempty.
Let $M_i$ be a sequence of properly embedded $m$-dimensional minimal submanifolds-with-boundary of $\Omega$
such that the $M_i$ converge as sets to a subset of $M$,
and such that the boundaries $\partial M_i$ converge smoothly to $\partial M$.
Then $M_i$ converges smoothly to $M$.

The result remains true if each $M_i$ is minimal with respect to a Riemannian metric $g_i$
provided the metrics $g_i$ converge smoothly to a limit Riemannian metric.
\end{theorem}

See \S\ref{varifold-boundary-reg-theorem} 
for a generalization to submanifolds $M_i$ of bounded mean curvature or (even more
generally) to varifolds with $(|H|-h)^+$ in $\LL^p$ for some $p>m$.

Note that we are not assuming any area bounds.  To deduce the smooth convergence $M_i\to M$
directly from Allard's Regularity Theorems (boundary and interior), one would need to assume that the $M_i$
converge weakly (in the sense of Radon measures) to $M$ with multiplicity $1$.  Indeed, we prove
Theorem~\ref{extended-allard-boundary-theorem} 
by deducing weak, multiplicity~$1$ convergence from the hypotheses.

\begin{proof}
The area blowup set of the $M_i$ is an $(m,0)$ set by 
 Theorem~\ref{main-theorem},
and it is contained in a connected $m$-manifold with nonempty boundary, so it is empty
by Corollary~\ref{constancy-corollary}.
That is, the areas of the $M_i$ are uniformly bounded on compact subsets of $\Omega$.
Thus by passing to a subsequence, we can assume that the $M_i$ converge
as varifolds to a varifold $V$ supported in $M$.  

Let $X$ be a compactly supported smooth
vectorfield on $\Omega$.
If we think of $M_i$ as a rectifiable varifold (by assigning it multiplicity $1$ everywhere), recall
that its first variation operator $\delta M_i$ is given by
\begin{align*}
  \delta M_i(X) 
  &= \int_{M_i}\Div_{M_i}X\,d\HH^m \\
  &= -\int_{M_i} H\cdot X\,d\HH^m + \int_{\partial M_i} X\cdot\nu_i\,d\HH^{m-1}.
\end{align*}

Thus
\begin{equation}
\begin{aligned}\label{key-boundary-inequality}
  | \delta M_i(X) |
   &\le \int_{M_i} |H\cdot X|\,d\HH^m + \int_{\partial M_i} |X\cdot\nu_i|\,d\HH^{m-1} \\
   &\le \int_{\partial M_i} |X|\,d\HH^{m-1}.
\end{aligned}
\end{equation}
Taking the limit as $i\to\infty$ gives
\begin{equation}\label{pre-riesz}
   |\delta V(X)| \le \int_{\partial M}|X|\,d\HH^{m-1}    
\end{equation}
In particular, $\delta V(X)=0$ for $X$ compactly supported in $\Omega\setminus\partial M$,
so by the constancy theorem for stationary varifolds~(\cite{allard-first-variation}*{\S4.6(3)} or
     \cite{simon-book}*{\S 41}), $V$ is the rectifiable varifold obtained by assigning some
constant multiplicity $a\ge 0$ to $M$.

         (Strictly speaking, 
         the constancy theorem only tells us that $V$ and the varifold $M$ with multiplicity $a$
        coincide in $\Omega\setminus \partial M$. However, since
         $V$ has locally bounded first variation (by~\ref{pre-riesz}), 
         $|V|$ is absolutely continuous
         with respect to $\HH^m$ (see \cite{simon-book}*{\S3.2, \S40.5}). Thus $|V|(\partial M)= 0$, 
         so in fact the two varifolds coincide throughout $\Omega$.)

Thus by the first variation formula for $M$, 
\[
  \delta V(X) =  a \int_{\partial M} X\cdot \nu\,d\HH^{m-1},
\]
where $\nu$ is the unit normal vectorfield to $\partial M$ that 
 points out of $M$.   Substituting this into~\eqref{pre-riesz} gives
\begin{equation}\label{force-inequality}
  a \int_{\partial M}X\cdot \nu \, d\HH^{m-1} \le  \int_{\partial M} |X| \,d\HH^{m-1}.
\end{equation}
Now let $X$ be a vectorfield whose restriction to $\partial M$ is $f\nu$, where $f$ is a nonnegative
function that is strictly positive on some nonempty open set.  Then~\eqref{force-inequality} becomes
\[
 a \int_{\partial M} f\,d\HH^{m-1} \le \int f\,d\HH^{m-1},
\]
which implies that $a\le 1$.

We have shown: the $M_i$ converge as varifolds to $M$ with multiplicity $a$ where $a\le 1$.
By Allard's Regularity and Boundary Regularity Theorems
 (or by the simplified version 
 in \cite{white-local-regularity}), the convergence is smooth on compact
 subsets of $\Omega$.    
 
 (Concerning the simplified versions of Allard's theorems: the proof described in~\cite{white-local-regularity} 
 is for interior points, but  the method works equally well at the boundary.)
\end{proof}

Readers interested in minimal (rather than bounded mean curvature) varieties may skip
the rest of this section.

Theorem~\ref{extended-allard-boundary-theorem} remains true if we replace the hypothesis that the $M_i$ are minimal 
by the hypothesis that
\[
    \limsup_{i\to\infty} \int_{K\cap M_i} ((|H|-h)^+)^p\,dA < \infty
\]
for every compact $K\subset\Omega$, provided we also replace smooth
convergence (in the conclusion) by convergence in ${\rm C}^1$ (with uniform local
$C^{1,1-m/p}$ bounds). 
However, 
the proof of Theorem~\ref{extended-allard-boundary-theorem} does not work
in the more general setting.  (As in the minimal case, by passing to a subsequence we
can assume that the $M_i$ converge as varifolds to limit varifold $V$ supported in $M$.
However, $V$ need not be stationary in $\Omega\setminus\partial M$, 
and thus we cannot invoke the constancy
theorem for stationary varifolds as we did in the minimal case.)  
So a different proof is required.  In fact, we 
prove a more general result that also applies to varifolds:

\begin{theorem}\label{varifold-boundary-reg-theorem}
Let $V_i$ be a sequence of $m$-dimensional varifolds in a smooth Riemannian manifold $\Omega$ such that
\begin{enumerate}[\upshape (1)]
\item\label{V-hypothesis} for each $i$
and for each smooth, compactly supported vectorfield $X$ on $\Omega$,
\begin{equation*}
  \delta V_i(X) = -\int X\cdot H_i\,d|V_i|  + \int_{\Gamma_i} X\cdot \nu_i \,d\HH^{m-1},
\end{equation*}
where $\Gamma_i$ is a smooth, proper,  $(m-1)$-dimensional submanifold of $\Omega$,
$H_i$ is a Borel vectorfield on $\Omega$, and $\nu_i$ is a Borel vectorfield on $\Gamma_i$
with $|\nu_i(x)| \le 1$ for all $x$.
\item\label{density-ge-1} $\Theta(V_i,x)\ge 1$ at each point of $\spt(V_i)\setminus \Gamma_i$.
\item The $\spt(V_i)$ converge as sets (see remark~\ref{remark:kuratowski})
to a subset $S$ of a connected, proper, ${\rm C}^1$ submanifold $M$
with $\partial M$ smooth and nonempty.
\item $\Gamma_i$ converges smoothy to $\partial M$.
\item\label{george} For every compact $K\subset \Omega$,
\[
       \limsup_{i\to\infty} \int_K (|H_i| - h)^+)^p\,d|V_i| < \infty,
\]
where $p$ and $h$ are finite constants with $p>m$.
\end{enumerate}
Then, after passing to a subsequence, the $V_i$ converge to a limit $V$.  
If $z$ is a point in $\partial M$, 
then $\Theta(V,z)=1/2$, and $z$ has a neighborhood $U$
such that:
\begin{enumerate}[\upshape(i)]
\item\label{boundary-conclusion-holder}
      for all sufficiently large $i$, the set $\spt(V_i)\cap U$ is a $C^{1,1-m/p}$ manifold-with-boundary in $U$
   (the boundary being $\Gamma_i\cap U$), 
   with a $C^{1,1-m/p}$ bound independent of $i$, and
\item\label{boundary-conclusion-C1-convergence} $\spt(V_i)\cap U$ converges in ${\rm C}^1$ to $M\cap U$.
\end{enumerate}
Furthermore, if $\beta>1$, then $U$ can be chosen so that
\begin{enumerate}[\upshape(i)]
\setcounter{enumi}{2}
\item\label{boundary-conclusion-density}  $\sup_{x\in U\setminus \Gamma_i} \Theta(V_i,x) \le \beta$ 
for all sufficiently large $i$.
\end{enumerate}
The theorem remains true if the $V_i$ satisfy the hypotheses for a sequence $g_i$ 
of Riemannian metrics on $\Omega$ converging smoothly to a limit metric $g$.
\end{theorem}

\begin{corollary}
Suppose that the $V_i$ in Theorem~\ref{varifold-boundary-reg-theorem}
 are integer-multiplicity rectifiable varifolds or,
more generally, varifolds with the gap $\alpha$  property (\S\ref{gap-alpha-definition})
for some $\alpha>1$ independent
of $i$.  
  Then $S=M$, and every point (interior or boundary)
of $M$
has a neighborhood $U\subset \Omega$ for which~\eqref{boundary-conclusion-holder}
and~\eqref{boundary-conclusion-C1-convergence} hold, and for which $\Theta(V_i,\cdot)\equiv 1$
on $\spt(V_i)\cap U\setminus \Gamma_i$ for all sufficiently large~$i$.
\end{corollary}

The corollary is false without the gap $\alpha$ assumption:
if we let $V_i$ be the portion of $M_i$ from \S\ref{counterexample-section} in the region $\{(x,y): |x|\le 5\}$
and if we let $M=[-5,5]\times\{0\}$,
then all the hypotheses
of Theorem~\ref{varifold-boundary-reg-theorem}
hold, but there are interior points $(x,y)$ of $M$ (namely the points $(\pm 1, 0)$) such that $(x,y)$ is a singular
point of every $V_i$.

\begin{proof}[Proof of corollary]
Let $z$ be a point in $\partial M$, let $\beta$ be a number such that $1<\beta<\alpha$.
Let $U$ be a neighborhood of $z$
satisfying the conclusions of the theorem.
Then by hypothesis~\eqref{density-ge-1} and by conclusion~\eqref{boundary-conclusion-density}, 
\[
     1 \le \Theta(V_i, x) < \alpha
\]
for all $x\in U\cap \spt(V_i)\setminus\Gamma_i$ and $i\ge i_0$, so by the gap $\alpha$ property, 
$\Theta(V_i,x)\equiv 1$ for such $x$ and $i_0$. 

Now by Theorem~\ref{more-extended-allard-theorem}, 
the multiplicity $1$ convergence in $U$ implies such convergence in all of 
$\Omega\setminus \partial M$.  But that implies that $S$ (the limit of the $\spt(V_i)$) is all
of $M$.  In particular, $S$ includes all of $\partial M$, so
 (by Theorem~\ref{varifold-boundary-reg-theorem}) we also get multiplicity $1$ convergence everywhere.
\end{proof}

\begin{proof}[Proof of Theorem~\ref{varifold-boundary-reg-theorem}] 
Let $h'=h+1$. Then, as in the proof of Theorem~\ref{more-extended-allard-theorem}, 
\[
       \limsup_{i\to\infty} \int_K (|H_i| - h')^+)\,d|V_i| < \infty,
\]
for every compact $K$.
Thus the area blowup set of the $V_i$ is an $(m,h')$-set by 
Theorem~\ref{main-theorem} and Definition~\ref{(m,h)-definition},
and it is contained in a connected $m$-manifold with nonempty boundary, so it is empty
by Corollary~\ref{constancy-corollary}.
In other words, the areas of the $V_i$ are uniformly bounded on compact sets.  It follows
(using hypothesis~\eqref{george} and Minkowski's inequality) that
\[
  \sup_K \int |H_i|^p\,d|V_i| < \infty
\]
for compact sets $K\subset \Omega$.  
By passing to a subsequence, we can assume that the $V_i$ converge to a varifold $V$.
Now let $z$ be a point in $S\cap \partial M$.  The remaining conclusions are local, so we can replace $\Omega$
by any open set containing $z$.  By isometrically embedding $\Omega$ into some $\RR^N$
and then enlarging it to get an open subset of $\RR^N$, we can assume that $\Omega$
is an open subset of $\RR^N$ with the Euclidean metric.  We may also assume that $z$ is the origin $0$.
By replacing $\Omega$ with an open ball whose closure is in $\Omega$, we can assume that
\begin{equation}\label{Lp-norm-finite}
  a:= \sup_i \left(\int |H_i|^p\,d|V_i|\right)^{1/p} < \infty.
\end{equation}
From the hypothesis~\eqref{V-hypothesis} and Holder's inequality, we have
\[
  |\delta V_i(X)| \le a \left(\int |X|^q\,d|V_i| \right)^{1/q} + \int_{\Gamma_i} |X|\,d\HH^{m-1}
\]
for all smooth, compactly supported vectorfields $X$, where $q=p/(p-1)$.
Passing to the limit gives
\begin{equation}\label{V-inequality}
  |\delta V(X)| \le a \left( \int |X|^q \,d|V| \right)^{1/q} + \int_{\partial M} |X|\,d\HH^{m-1}.
\end{equation}

For $r>0$, let $V^r$, $M^r$, and $\Omega^r$ be obtained from $V$, $M$, and $\Omega$ by dilation by $1/r$ 
about $0$.

We claim that
\begin{equation}\label{rescaling}
   \delta V^r(X) \le r^{1-m/p}a \left(\int |X|^q \, d|V^r|\right)^{1/q} +  \left( \int_{\partial M^r} |X|\,d\HH^{m-1} \right)
\end{equation}
for every smooth vectorfield $X$ supported in $\Omega^r$.
To prove the claim, fix an $r$ and let $\tilde X(x)=X(x/r)$.
Then
\begin{align*}
|\delta V^r(X)|
&=
r^{1-m} |\delta V(\tilde X)| 
\\
&\le
r^{1-m}a \left(\int |\tilde X|^q \,d|V|\right)^{1/q} + r^{1-m}  \left(\int_{\partial M} |\tilde X| \,d\HH^{m-1}\right)
\\
&=
r^{1-m/p}a \left(\int |X|^q \, d|V^r|\right)^{1/q} +  \left( \int_{\partial M^r}  |X|  \,d\HH^{m-1} \right).
\end{align*}
This proves the claim.

Let 
\begin{equation}
  \theta:
  = \Theta^*(|V|,0) \\
  := \limsup_{r\to 0} \frac{|V|\,\BB(0,r)}{\omega_mr^m} \\
  = \limsup_{r\to 0}  \frac{|V^r|\,\BB(0,1)}{\omega_m}
  \in [0,\infty].
\end{equation}
Consider a sequence of $r$'s tending to $0$ and let $\Lambda$ be the set of those $r$'s.
Choose $\Lambda$ so that
\begin{equation}\label{V-normalized}
   \lim_{r\in \Lambda\to 0} \frac{|V^r|\,\BB(0,1)}{\omega_m} = \theta.
\end{equation}
By passing to a further subsequence, we can assume that the supports of $V^r$ converge as $r\in\Lambda\to 0$
to a subset of $M':=\Tan(M,0)$, the tangent halfplane to $M$ at $0$.
Thus by~\eqref{rescaling} and Corollary~\ref{constancy-corollary}, 
the areas of the $V^r$ are uniformly bounded on compact sets, so, by passing
to a further subsequence, we can assume that the $V^r$ converge to a limit varifold $V'$ as $r\in\Lambda\to 0$.
From~\eqref{rescaling}, we see that
\begin{equation}\label{V'-inequality}
  \delta V'(X) \le   \int_{\partial M'} |X|  \,d\HH^{m-1} 
\end{equation}
for all smooth, compactly supported $X$.  In particular, $V'$ is stationary in $\RR^N\setminus \partial M'$,
so, by the constancy theorem for stationary varifolds~(\cite{allard-first-variation}*{\S 4.6(3)} or
     \cite{simon-book}*{\S41}),
$V'$ is the halfplane $M'$ with some constant multiplicity.  
By~\eqref{V-normalized}, that multiplicity is $2\theta$.  
Thus
\[
  \delta V'(X) = 2\theta \int \Div_{M'}X\,d\HH^m = 2\theta \int_{\partial M'} X\cdot \nu\,d\HH^{m-1},
\]
where $\nu$ is the unit normal vector to $\partial M'$ that points out from $M'$.  
Thus by~\eqref{V'-inequality},
\[
2\theta \int_{\partial M'} X\cdot \nu\,d\HH^{m-1} \le \int_{\partial M'} |X| \, d\HH^{m-1},
\]
which immediately implies that $\theta \le 1/2$.  (Let $X$ be a smooth, compactly supported
vecforfield whose restriction to $M'$ is $f\nu$, where $f$ is a nonnegative function that is not
identically $0$.)
 
Now $\theta=\Theta^*(|V|,0)\le 1/2$ implies, for all sufficiently small balls $\BB(0,r)$, that 
$V_i\llcorner \BB(0,r)$ satisfies the hypotheses of the Allard
boundary regularity Theorem~\cite{allard-boundary}*{p.~429} for all sufficiently large $i$, which
implies the asserted behavior~\eqref{boundary-conclusion-holder} and~\eqref{boundary-conclusion-C1-convergence}
 in a smaller ball.  Also, hypothesis~\eqref{density-ge-1} and conclusion~\eqref{boundary-conclusion-C1-convergence}
 of the theorem imply that
 $\Theta(V ,0)\ge 1/2$.  Therefore $\Theta(V,0)=1/2$.

It remains only to prove~\eqref{boundary-conclusion-density}. 
Let $U$ satisfy~\eqref{boundary-conclusion-holder} and~\eqref{boundary-conclusion-C1-convergence}. 
We may assume that~\eqref{boundary-conclusion-holder} and~\eqref{boundary-conclusion-C1-convergence} 
hold for all $i$ by dropping the first $i_0$ terms in the sequence.
Now suppose that~\eqref{boundary-conclusion-density} does not hold for any $U$.
Then, after passing to a subsequence, we can assume that there
are points $x_i\in  U\setminus \Gamma_i$ such that $x_i\to 0$ and such that
\[
   \Theta(V_i, x_i) \ge \beta.
\]
Let $y_i$ be the point in $\Gamma_i$ nearest to $x_i$.  Translate $V_i$, $M$, and $x_i$ by $-y_i$
and dilate by $1/|x_i-y_i|$ to get $V_i^\dag$, $M_i^\dag$ and $x_i^\dag$.  
Note that the $M_i^\dag$ converge to the halfplane $M'=\Tan(M',0)$.  (This follows
from the ${\rm C}^1$ convergence in~\eqref{boundary-conclusion-C1-convergence}.)   Now by exactly the same reasoning
used for the $V^r$, we can assume, after passing to a subsequence, that the $V_i^\dag$ converge
to a limit $V^\dag$ consisting of the halfplane $M'=\Tan(M,0)$ with some constant multiplicity $c\le 1$.
Note that the points $x_i^\dag$ converge to the point $x^\dag$ in $M'$ such that $x^\dag$ is 
a unit vector in $M'$ perpendicular to $\partial M'$.  
Now
\[
   1 \ge c = \Theta(V^\dag ,x^\dag) \ge \limsup_i\Theta(V_i^\dag, x_i^\dag) \ge \beta
\]
by the upper semicontinuity of density for varifolds whose mean curvatures 
satisfy uniform local $\LL^p$ bounds~\cite{simon-book}*{\S17.8}. However, $\beta>1$ by hypothesis.  
The contradiction proves~\eqref{boundary-conclusion-density}.
\end{proof}

\begin{remark}
In Theorem~\ref{varifold-boundary-reg-theorem}, the hypothesis that $|\nu(\cdot)|\le 1$
can be relaxed $|\nu(\cdot)| \le \gamma$, where $\gamma>1$ is a constant (depending on $m$ and on $\dim(\Omega)$)
from the Allard Boundary Regularity Theorem.  If the $V_i$ have the gap $\alpha$ property, then we can let $\gamma$
be any number with $1<\gamma<\alpha$.   
The proof is almost exactly the same as the proof of Theorem~\ref{varifold-boundary-reg-theorem}. 
\end{remark}

\section{The Barrier Principle}\label{barrier-section}

The following theorem shows that an $(m,h)$ subset
obeys the same barrier form of the maximum principle that is satisfied by
smooth $m$-manifolds
with mean curvature bounded by $h$.

\begin{theorem}[Barrier Principle]\label{barrier-theorem}
Let $\Omega$ be a ${\rm C}^1$ Riemannian manifold without boundary,
and let $Z$ be
an $(m,h)$ subset of $\Omega$.
Let $N$ be a closed region in $\Omega$ with smooth boundary such that $Z\subset N$,
and let $p\in Z\cap \partial N$.  Then
\[
  \kappa_1 + \dots + \kappa_m \le h
\]
where $\kappa_1\le \kappa_2 \le \dots \le \kappa_{n-1}$ are the principal curvatures
of $\partial N$ at $p$ with respect to the unit normal that points into $N$.
\end{theorem}

\begin{proof}[Proof of the Barrier Principle (Theorem~\ref{barrier-theorem})]
Since the result is local, we may assume that $\Omega$ is an open subset of $\RR^n$.
It suffices to construct a smooth function $f:\Omega\to \RR$
such that 
such that 
\begin{align}
\label{maxN}
\max_N f &= f(p), \\
\label{p-noncritical}
\D f(p) &\ne 0, 
\end{align}
and such that
\[
   \frac{ \trm(\D^2f)(p)}{ |\D f(p)|}  = \mu:=\sum_{i=1}^m\kappa_i
\]

Case 1: $g$ is the Euclidean metric.  Let $u:\Omega\to \RR$
be the signed distance to $\partial N$:
\begin{equation*}
u(x)
=
\begin{cases}
\dist(x, \partial N) &\text{if $x\notin N$},
\\
-\dist(x, \partial N) &\text{if $x\in N$}.
\end{cases}
\end{equation*}
Let $\ee_1, \dots, \ee_{n-1}$ be unit vectors in $\Tan_p\partial N$ in the principal
directions of $\partial N$.  These vectors together with $\nabla u(p)$ form an orthonormal
basis for $\RR^n$, and a standard and straightforward computation shows that these
are eigenvectors of $\D\nabla u(p)$ with eigenvalues $\kappa_1, \dots, \kappa_{n-1}$, and $0$.

Let $f(x)=e^{\alpha u(x)}$, where $\alpha$ is a positive number to specified later.  Then
\[
  \D f = \alpha e^{\alpha u} \D u.
\]
If we work in normal coordinates at $p$ and (by a slight abuse of notation) use $\D^2 f$
to denote the matrix of second partial derivatives of $f$ with respect to those coordinates,
then we have (at the point $p$)
\[
  \D^2 f = \alpha^2 e^{\alpha u}\D u^T \D u  + \alpha e^{\alpha u} \D^2  u.
\]
From this we see that the eigenvectors of $\D^2u(p)$ are also eigenvectors of $\D^2f(p)$,
and that the eigenvalues of $\D^2f(p)$ are 
\begin{equation}\label{eigs}
  \lambda_i = \alpha \kappa_i  \quad (i=1,\dots, n-1)
\end{equation}
together with $\lambda_n := \alpha^2$.  
Choosing $\alpha$ so that
\[
  \alpha > \max_i  | \kappa_i |
\]
guarantees that $\lambda_n$ is the largest eigenvalue and thus by~\eqref{eigs} that
\[
  \trm (\D^2f(p)) = \sum_{i=1}^m \alpha\kappa_i = \alpha\mu,
\]
so
\[
   \frac{\trm(\D^2f(p))}{|\D f(p)|} = \frac{\alpha \mu}{\alpha} = \mu.
\]
This completes the proof in case 1.

Case 2: $g$ is a general ${\rm C}^1$ metric.
As before, we can assume that $\Omega\subset \RR^n$.
By a diffeomorphic change of coordinates, we may assume that
\begin{equation}\label{normal1}
   g_{ij}(p) = \delta_{ij}
\end{equation}
and that 
\begin{equation}\label{normal2}
    \D g_{ij}(p) = 0.
\end{equation}

Now by~\eqref{normal1} and~\eqref{normal2}, at the point $p$,
the principal curvatures of $\partial M$ with respect to the Euclidean
metric $\delta$ are equal to the principal curvatures with respect to the metric $g$.
Thus, by case 1, there is a smooth function $f:\Omega\to \RR$ such that
$\D f(p)\ne 0$, 
\[
  \max_N f = f(p),
\]
and such that
\begin{equation}\label{equals-mu}
  \left[ \frac{\trm(\D^2f(p))}{|\D f(p)|} \right]_\delta = \mu.
\end{equation}
But by~\eqref{normal1} and~\eqref{normal2}, the left side of~\eqref{equals-mu}
 does not change if we replace $\delta$ by $g$.
This completes the proof in case 2.
\end{proof}

\begin{corollary}\label{avoidance-corollary}
Suppose $Z$ is an $(m,0)$ subset of smooth Riemannian $(m+1)$-manifold.
If $t\in [0,T] \mapsto M(t)$ is a mean curvature flow of compact hypersurfaces 
and if $M(0)$ is disjoint from $Z$, then $M(t)$ is disjoint from $Z$ for all $t\in [0,T]$.
\end{corollary}

Here the mean curvature flow can be a classical flow, a Brakke flow of varifolds,
or a level-set flow.
See~\cite{white-isoperimetric}*{Proposition~7.7} for the proof.
(There $Z$ is stated to be the support of a stationary $m$-varifold, but
in fact the proof only uses the Barrier Principle~\ref{barrier-theorem}
and hence establishes Corollary~\ref{avoidance-corollary} for any $(m,0)$ set $Z$.)

In the codimension one case, we also have a strong barrier principle:

\begin{theorem}[Strong Barrier Principle]\label{strong-barrier-theorem}
Let $Z$ be an $(m,h)$ subset of a smooth,
$(m+1)$-dimensional, Riemannian manifold $\Omega$ without boundary.

Let $N$ be a closed region in $\Omega$ with smooth, connected boundary such that $Z\subset N$
and such that
\[
    H_{\partial N}\cdot \nu \ge h
\]
at every point of $\partial N$, where $H_{\partial N}(x)$ is the mean curvature vector of $\partial N$ at $x$
and $\nu(x)$ is the unit normal at $x$ to $\partial N$ that points into $N$.

If $Z$ contains any points of $\partial N$, then it contains all of $\partial N$.
\end{theorem}

\begin{proof}
See~\cite{SolomonWhite} for a proof.  Specifically,
 \cite{SolomonWhite}*{step 1, page 687} shows that any set $Z$ that violates the conclusion of
 the strong barrier principle~\ref{strong-barrier-theorem} also violates the conclusion
 of the barrier principle~\ref{barrier-theorem}.
 (The proof there is written for the case $h=0$, but the same proof works for arbitrary $h$.)
 \end{proof}
 
\begin{corollary}[The Halfspace Theorem for $(2,0)$ sets]\label{halfspace-theorem-1}
Suppose $Z\subset \RR^3$ is a nonempty $(2,0)$ set that lies in a halfspace of $\RR^3$.
Then $Z$ contains a plane.  Indeed, if $L:\RR^3\to\RR$ is a nonconstant linear function
and if 
\[
    s:=\sup_ZL<\infty,
\]
 then $Z$ contains the plane $L=s$.
\end{corollary}

\begin{proof} 
Hoffman and Meeks~\cite{hoffman-meeks-halfspace}*{Theorem~1} proved this in case $Z$ is a properly immersed minimal submanifold of $\RR^3$,
but their proof only uses the strong barrier principle and hence also works for arbitrary $(2,0)$ sets $Z$.
\end{proof}

\section{Converse to the Barrier Principle}\label{converse-section}

\begin{lemma}\label{gradient-nonzero-lemma}
Suppose $Z$ is a closed subset of a Riemannian manifold $\Omega$.
If $Z$ is not an $(m,h)$ set, then there is smooth function $f:\Omega\to \RR$ such
that $f|Z$ has a local maximum at a point $p$ where 
\begin{equation}\label{defining-inequality}
   \trm(\D^2f(p)) > h\, |\D f(p)|
\end{equation}
and where
\begin{equation}
     \D f(p)\ne 0.
\end{equation}
\end{lemma}

\begin{proof}
Since the result is local, we may assume that $\Omega$ is diffeomorphic to a ball or,
equivalently, to $\RR^n$.  Thus we may in fact assume that $\Omega$ is $\RR^n$ with
a Riemannian metric.  By hypothesis, there is a ${\rm C}^2$ function $f:\Omega\to \RR$
and a point $p$ such that $f|Z$ has a local maximum at $p$ and such that~\eqref{defining-inequality} holds.
By corollary~\ref{special-f-corollary}, there is such an $f$ that is smooth.
By translation, we may assume that $p=0$.

We assume that $\D f(0)=0$, as otherwise there is nothing to prove.
By replacing $f$ by 
\[
   x \mapsto f(x) - \eps\,|x|^2
\]
for a sufficiently small $\eps>0$, we may assume that $f|Z$ has a strict local maximum at $0$
and that $\D f$ has an isolated zero at $0$, i.e., that 
\begin{equation}\label{punctured}
   \text{ $\D f(x)\ne 0$ if $0< |x| < r$ }
\end{equation}
for some $r>0$.

Since $\trm(\D^2f(0))>0$, the function $f$ does not have a local maximum at $0$.  Thus $0$ is not
in the interior of $Z$.  Let $p_i$ be a sequence of points in $\Omega\setminus Z$ converging to $0$.
Let
\begin{align*}
&f_i: \Omega\to \RR \\
&f_i(x) = f(x-p_i).
\end{align*}

Since $f|Z$ has a strict local maximum at $p$, it follows that (for sufficiently large $i$)
$f_i|Z$ has a local maximum at some point $q_i$ with $\lim_i q_i=0$. By the smooth convergence
$f_i\to f$ and by~\eqref{defining-inequality}, 
\[
    \trm(\D^2f_i(q_i)) - h\,|\D f_i(q_i)| > 0
\]
for all sufficiently large $i$.

Now $|q_i-p_i|>0$ since
 since $q_i\in Z$ and $p_i\notin Z$.  Also, $|q_i-p_i|\to 0$ since $p_i$ and $q_i$ tend to $0$.
Thus $|Df_i(q_i)|\ne 0$ by~\eqref{punctured}.

Thus (for all sufficiently large $i$) the function $f_i$ and the point $q_i$ have the desired properties.
\end{proof}

\begin{theorem}\label{converse-theorem}
Let $Z$ be a closed subset of a Riemannian manifold $\Omega$ and let $m<\dim(\Omega)$.
Suppose that $Z$ is not an $(m,h)$ set.  
Then there is a closed region $N\subset \Omega$ containing $Z$ and
a point $p\in Z\cap \partial N$ such that $\partial N$ is smooth and such that
\[
     H_m(\partial N, p) > h
\]
where $H_m(\partial N,p)$ is the sum of the smallest $m$ principal curvatures of $\partial N$ at $p$
with respect to the unit normal that points into $N$.
\end{theorem}

\begin{proof}
By hypothesis, there is a point $p\in Z$ and a smooth function $f:\Omega\to \RR$
such that $f|Z$ has a local maximum at $p$ and such that 
\[
  \trm(\D^2f(p)) >  h \, |\D f(p)|.
\]
By Lemma~\ref{gradient-nonzero-lemma}, we may assume that $\D f(p)\ne 0$.  We may also assume
that 
\[
   |\D f(p)|=1.
\]
(Otherwise replace $f$ by $f/ |\D f(p)|$.)
Thus
\[
  \trm(\D^2f(p)) > h.
\]
 By modifying $f$ outside of a compact
neighborhood of $p$, we may assume that $f|Z$ attains its global maximum at $p$, 
and that $\D f$ never vanishes on the level set $f=f(p)$.  Hence the set
$N:=\{x: f(x)\le f(p)\}$ is a closed region with smooth boundary, $Z\subset N$,
and $p\in Z\cap \partial N$.

Let 
\[
 \kappa_1 \le \kappa_2 \le \dots \le \kappa_{n-1}
\]
be the principal curvatures of $\partial N$ at $p$ with respect to the unit normal that points into $N$.
We may suppose that we have chosen normal coordinates at $p$ such that
the standard basis vectors 
 $\ee_1, \ee_2, \dots, \ee_{n-1}$ 
are the corresponding principal directions of $\partial N$ at $p$.
Let 
\[
   \nu = \frac{\nabla f}{|\nabla f|}
\]
and $s=|\nabla f|$, so that $\nabla f = s\nu$.
Now
\begin{align*}
   \D_{\ee_i}(\nabla f(p)) 
   &= \D_{\ee_i} s\nu \\
   &= s \D_{\ee_i}\nu + \nu \D_{\ee_i}s \\
   &= \kappa_i\ee_i + \nu \D_{\ee_i}s,
\end{align*}
so
\[
  \ee_i\cdot \D_{\ee_i}\nabla f(p) = \kappa_i.
\]
In other words, $\kappa_i$ is the $ii$ entry of the matrix for $\D\nabla f(p)$ with respect to the orthonormal
basis $\ee_1, \dots, \ee_n$.
Thus
\[
 h <  \trm(\D^2f(p)) \le \sum_{i=1}^m\kappa_i.
\]
(For the last step we are using the following fact from linear algebra: 
if $Q$ is a symmetric $n\times n$ matrix, then the sum of any $m$ of the diagonal
entries of $Q$ is greater than or equal to the sum of the smallest $m$ eigenvalues of $Q$.)
\end{proof}

\section{Minimal Hypersurfaces}\label{minimal-hypersurfaces-section}

Here we prove some results in the special case of $(m,0)$ sets in $(m+1)$-dimensional manifolds.
In the next section, we extend the results to $(m,h)$ sets with $h>0$.

We suppose throughout this section that $N$ is a smooth, $(m+1)$-dimensional Riemannian manifold with smooth, connected boundary.  We also suppose that one of the following hypotheses holds:
\begin{enumerate}
\item\label{ricci} $N$ is complete with Ricci curvature bounded below, or
\item\label{exhaust} $N$ has an exhaustion by nested, compact, mean convex regions, or
\item\label{piecewise} $N$ is a subset of a larger $(m+1)$-manifold and $\overline{N}$ is compact
and mean convex.
\end{enumerate}

(Each of these hypotheses guarantees that a compact hypersurface in $N$ moving by mean curvature flow
cannot escape to infinity in finite time.  For hypotheses~(2) and~(3), this follows immediately from the maximum principle.  For hypothesis~\eqref{ricci}, see~\cite{ilmanen-generalized-flow-of-sets}*{6.2 or 6.4}.  
In hypotheses~\eqref{exhaust} and~\eqref{piecewise}, 
the mean convex regions referred to need not have smooth boundary.)

\begin{theorem}\label{minimal-hypersurface-theorem}
Let $m<7$ and let $N$ be a smooth, mean convex, $(m+1)$-dimensional Riemannian manifold with smooth, nonempty, connected boundary satisfying one of the hypotheses (1)--(3) above.

Suppose that $N$ contains a nonempty $(m,0)$ subset $Z$ and that $Z$ does not contain all of $\partial N$.

Then $N$ contains a nonempty, smooth, embedded, 
stable minimal
hypersurface $S$ that weakly separates $Z$ from $\partial N$ in the following sense:
 if $C\subset N$ is a connected, compact set that contains points of $Z$ and of $\partial N$, 
 then $C$ intersects $S$.

The theorem remains true for $m\ge 7$, except that the surface $S$ is allowed to have a singular
set of Hausdorff dimension $\le m-7$.
\end{theorem}

(We remark that $S$ has a one-sided minimizing property considerably stronger than stability.
See~\cite{white-size}*{\S11} for details.
In particular, if any connected component of $S$ is one-sided (i.e., has a nonorientable normal bundle),
then its two-sided double cover is also stable.)

\begin{proof}
By the Strong Barrier Principle~\ref{strong-barrier-theorem}, the set $Z$ must lie in the interior of $N$.
If $\partial N$ is a stable minimal hypersurface, then we let $S=\partial N$.
Thus we may assume that $\partial N$ is not a minimal hypersurface or that
it is an unstable minimal hypersurface.  We divide the proof into four cases
according to whether $\partial N$ is or is not minimal and whether it is or
is not compact.

{\bf Case 1}: $\partial N$ is compact and $\partial N$ is not a minimal surface.   
Let 
\[
   t\in [0,\infty) \mapsto K(t)
\]
be the flow such that $K(0)=N$ and such that $\partial K(t)$ flows by mean curvature flow.
Each of the hypotheses~\eqref{ricci}, \eqref{exhaust}, and \eqref{piecewise}
 imply that $\partial K(t)$ remains in $N$ (as a compact set) for all time.  

Also, since $Z$ is an $(m,0)$ set, $\partial K(t)$ can never bump into $Z$ (Corollary~\ref{avoidance-corollary}.)
That is, $Z$ is contained in the interior of $K(t)$ for all $t$.
Thus $Z\subset K_\infty \subset \text{interior}(N)$ where $K_\infty=\cap_tK(t)$.
Furthermore, by~\cite{white-size}*{\S11}, 
$S:=\partial K_\infty$ is a minimal surface with the indicated regularity properties.
This completes the proof in case~1.

{\bf Case 2}: $\partial N$ is a compact, unstable minimal hypersurface.  The instability means that
we can push $\partial N$ slightly into $N$ 
to get a surface whose mean curvature is everywhere nonzero and points away from $\partial N$.
(For example, we can push $\Sigma$ into $N$ by the lowest eigenfunction
of the Jacobi operator; see \cite{hoffman-white-genus-one}*{Proposition~A3} for a proof.)  
Replacing $N$ by the portion of $N$ on one side of that surface reduces case 2 to case 1.

{\bf Case 3}: $\partial N$ is noncompact and nonminimal.  In this case, let $p$ be a point in $\partial N$
where the mean curvature of $\partial N$ is nonzero.  Let $f:\partial N\to \RR$
be a proper Morse function such that $f(p)=\min f < 0$ and such that $0$
is a regular value of $f$.   
Let 
\[
  t\in[0,\infty)\mapsto K(t) \subset N
\]
be the flow such that 
\begin{gather*}
K(0)=N,\\
K(t)\cap (\partial N)= \{q\in \partial N: f(q)\ge t\},
\end{gather*}
and such that the surfaces
\[
   M(t):= \partial K(t)
\]
move by mean curvature flow.

(Note that $M(t)$ is a (possibly singular) $m$-dimensional surface with boundary,
the boundary of $M(t)$ is $\{q\in \partial N: f(q)=t\}$.)

The rest of the proof is essentially identical to the proof in case 1.

{\bf Case 4}: $\partial N$ is a noncompact, unstable minimal hypersurface.
Let $f:\partial N\to \RR$ be a smooth, proper Morse function that is bounded below.
Since $\partial N$ is unstable, it follows that
 for all sufficiently large $t$, the surface $(\partial N)\cap \{f<t\}$ will be unstable.
In particular, there is a regular value $\tau$ of $t$ for which the surface $\Sigma:=(\partial N)\cap\{f<\tau\}$
is unstable.  
By adding a constant to $f$, we may suppose that $\tau=0$.
The instability implies that 
we can push the interior of $\Sigma$ slightly into the interior of $N$ to get a surface $\Sigma'$
with $\partial \Sigma'=\partial \Sigma$ such that the mean curvature of $\Sigma'$ is everywhere nonzero and points away from $\partial N$.  For example, we can push $\Sigma$ into $N$ by the lowest eigenfunction
of the Jacobi operator as in case (2).
We make the perturbation small enough that the closed region bounded by $\Sigma\cup\Sigma'$ does
not contain any points of $Z$.

Now let 
\[
  t\in [0,\infty)\mapsto M(t)
\]
be the mean curvature flow (constructed by elliptic regularization) such that $M(0)=\Sigma'$
and such that $\partial M(t)= \{x\in \partial N: f(x)=t\}$ for all $t\ge 0$.

The rest of the proof is identical to the proof in case 3.
\end{proof}

\begin{corollary}[Strong Halfspace Theorem for $(2,0)$ sets]\label{halfspace-theorem-2}
Let $\Sigma$ be a connected, properly embedded, separating minimal surface in a 
complete $3$-manifold $\Omega$
of nonnegative
Ricci curvature.
Suppose $Z$ is a nonempty $(2,0)$ set that lies in the closure $N$ of one of the connected
components of $\Omega\setminus \Sigma$, and suppose that $\Sigma\setminus Z$ is nonempty.  
Then $N$ contains
a properly embedded, totally geodesic surface $M$ with Ricci flat normal bundle.

In particular, if $\Omega$ is the flat $\RR^3$, then $\Sigma$ is a plane and $Z$ contains
a plane parallel to $\Sigma$.
\end{corollary}

Hoffman and Meeks~\cite{hoffman-meeks-halfspace}*{Theorem~2} 
proved this in case $Z$ is a properly immersed minimal surface.

\begin{proof}
The corollary follows from the Theorem~\ref{minimal-hypersurface-theorem}
 because by~\cite{fischer-colbrie-schoen}*{page 210, paragraph 1},
every complete, stable, two-sided
minimal surface $M$ in $\Omega$ is totally geodesic and has Ricci flat normal bundle.

The last assertion (``$Z$ contains a plane parallel to $\Sigma$'') is Corollary~\ref{halfspace-theorem-1}.
\end{proof}

\section{Bounded Mean Curvature Hypersurfaces}\label{bounded-mean-curvature-section}

Here we extend Theorem~\ref{minimal-hypersurface-theorem} from  $(m,0)$ sets to $(m,h)$ sets.

\begin{definition}
Let $N$ be a smooth Riemannian manifold with smooth boundary and let $h\ge 0$.  We say that $N$ is 
{\em $h$-mean convex} provided 
\begin{equation}\label{h-convexity-inequality}
    H_{\partial N}\cdot \nu\ge h
\end{equation}
at all points of $\partial N$, where $H_{\partial N}$ is the 
mean curvature vector and $\nu$ is unit normal to $\partial N$ that points into $N$.
\end{definition}

It is also convenient to allow $N$ with piecewise smooth boundary.
In particular, suppose $N=\cap_iN_i$ is the intersection of finitely many smooth Riemannian
manifolds with smooth boundary and that the $\partial N_i$ are transverse.
(The transversality means that if $x$ belongs to several of the $\partial N_i$, then the unit
normals to those $\partial N_i$ at $x$ are linearly independent.) 
In that case, we say that $N$ is {\em $h$-mean convex} provided~\eqref{h-convexity-inequality} holds at all the regular
boundary points of $N$.

In this section, we suppose that $h>0$ and that $N$ is a smooth, $(m+1)$-dimensional Riemannian manifold 
that satisfies one of the following hypotheses:
\begin{enumerate}[\upshape (i)]
\item\label{ricci2} $N$ is complete with Ricci curvature bounded below.
\item\label{exhaust2} $N$ has an exhaustion by nested, compact, $h$-mean convex regions.
\item\label{piecewise2} $N$ is a subset of a larger $(m+1)$-manifold and $\overline{N}$ is compact
and $h$-mean convex.
\end{enumerate}

(The exhausting regions in~\eqref{exhaust2} and the region $\overline{N}$ in~\eqref{piecewise2}
  are allowed to have piecewise smooth boundary.)

\begin{theorem}\label{bounded-mean-curvature-hypersurface-theorem}
Let $h>0$,  $m<7$, and let $N$ be a smooth, $h$-mean convex, $(m+1)$-dimensional Riemannian manifold with smooth, nonempty, connected boundary.
Suppose that one of the hypotheses (i), (ii) or (iii) holds, and 
that $N$ contains a nonempty $(m,h)$ subset $Z$.

Then $Z$ is contained in a region $K$ whose boundary is smooth and has constant mean curvature $h$ with respect
to the inward unit normal.  Furthermore, if $\partial N$ is not contained in $Z$, then $\partial K$ is stable
for the functional 
$({\rm area}) - h({\rm enclosed\,volume})$.

The theorem remains true for $m\ge 7$, except that the surface $S$ is allowed to have a singular
set of Hausdorff dimension $\le m-7$.
\end{theorem}

\begin{proof}
The proof is exactly the same as the proof of Theorem~\ref{minimal-hypersurface-theorem},
except that in that proof, we let the sets $K(t)$ evolve so that $\partial K(t)$ moves
not with velocity $H$ but rather with velocity $H-h\nu$ where $H$ is the mean curvature
and $\nu(x)$ is the inward unit normal.

Suitable varifold solutions to the flow can be constructed by elliptic regularization just as in the $h=0$ case.
Furthermore, $h$-mean convexity is preserved by the flow just as in the $h=0$ case.  Indeed,
all the results in~\cite{white-size} for mean convex mean curvature flow continue to hold for arbitrary $h$, with only
very minor modifications in the proofs.   In fact, for $h>0$, the behavior of $\partial K(t)$ as $t\to\infty$
is slightly simpler: in the case $h=0$, it is possible for $\partial K(t)$ to converge smoothly to a double
cover of the limit surface $S$, whereas for $h>0$, that is clearly impossible.
\end{proof}

\section{The Distance to an $(m,h)$ Set}\label{distance-section}

Here we show that $(m,h)$ sets behave well with respect to the distance function.
The theorem and its proof are particularly simple when the ambient space is Euclidean,
so we consider that case first:

\begin{theorem}
Suppose $Z$ is an $(m,h)$ subset of $\RR^n$.  Then for $s>0$, the set $Z(s)$
of points in $\RR^n$ at distance $\le s$ from $Z$ is also an $(m,h)$ set. 
\end{theorem}

\begin{proof}
Let $f:\RR^n\to \RR$ be a smooth
function such that $f|Z(s)$ has a local maximum at $p\in Z(s)$.
Let $q$ be a point in $Z$ that minimizes $\dist(q,p)$.   Let
\[
    g(x) = f(x + p - q).
\]
Then $g|Z$ has a local maximum at $q$, so
\[
  \trm(\D^2g(q)) - h\,|\D g(q)| \le 0.
\]
Since $\D f(p)=\D g(q)$ and $\D^2f(p)=\D^2g(q)$, this implies that
\[
  \trm(\D^2f(p)) - h\,|\D f(p)| \le 0.
\]
\end{proof}

\begin{theorem}
Suppose $Z$ is an $(m,h)$ subset of a connected, Riemannian manifold $\Omega$.
For $s>0$, let $Z(s)$ be the set of points at geodesic distance $\le s$ from $Z$.
\begin{enumerate}[\upshape (i)]
\item\label{sectional-case} If the sectional curvatures of $\Omega$ are bounded below by $K$,
then $Z(s)$ is an $(m, h - mKs)$ set.
\item\label{ricci-case} If $\dim(\Omega)=m+1$ and if the Ricci curvature of $\Omega$ is bounded
below by $\rho$, then $Z(s)$ is an $(m, h - \rho s)$ set.
\end{enumerate}
\end{theorem}

\begin{proof}
If $\dim(\Omega)=m$, then by the constancy Theorem~\ref{constancy-theorem}, 
 $Z$ is either all of $\Omega$ or the empty set, in either of which cases
the theorem is trivially true.   Thus we suppose that $\dim(\Omega)>m$.

Let $N$ be a closed region in $\Omega$ with smooth boundary such that $Z(s)\subset N$
and such that $p\in Z(s)\cap \partial N$.  By Theorem~\ref{converse-theorem}, it suffices to show that
\[
    H_m(\partial N, p) \le h - mKs.
\]
in case \eqref{sectional-case} or
\[
   H_m(\partial N,p)\le h- \rho s
\]
in case \eqref{ricci-case}.
Let $q$ be a point in $Z$ such that $\dist(q,p)=s$.   Let $\Gamma$ be the geodesic joining $p$ to $q$.
Note that the signed distance function $\dist(\cdot, \partial N)$ will be smooth on an open set containing $\Gamma\setminus\{q\}$, but that it may not be smooth at $q$.

We get around that lack of smoothness 
as follows.  Note that for each $\eps>0$, we can find a closed region $N'\subset \Omega$ with
smooth boundary such that 
\begin{enumerate}[\upshape (1)]
\item\label{containment} $N\subset N'$, 
\item\label{touching} $N\cap \partial N'=\{p\}$, 
\item
 the principal directions of $\partial N$ at $p$ are also principal directions of $\partial N'$ at $p$, 
\item\label{non-osculating}
  each principal curvature of $\partial N'$ at $p$ is strictly less than the corresponding principal curvature of 
  $\partial N$ at $p$,
\item\label{curvatures-close}
 each principal curvature of $\partial N'$ at $p$ is within $\eps$ of the the corresponding principal
curvature of $N$ at $p$.
\end{enumerate}
By~\eqref{curvatures-close},
\begin{equation}\label{H_ms-close}
H_m(\partial N',p) \ge H_m(\partial N, p) -m\eps.
\end{equation}
By~\eqref{non-osculating}, the function $f(\cdot):=\dist(\cdot, \partial N')$
 is smooth on an open subset of $N'$ containing $\Gamma$.
In particular, if 
\[
   N^* = \{x\in N': \dist(x,\partial N')\ge s\}
\]
then $q\in \partial N^*$ and $\partial N^*$ is smooth near $q$.

It follows (Proposition~\ref{eigenvalue-proposition}) that each principal curvature of $\partial N^*$ at $q$
is greater than or equal to $Ks$ plus the corresponding principal curvature of $\partial N'$ at $p$,
and thus (taking the sum of the first $m$ principal curvatures) that
\[
  H_m(\partial N^*,q)  \ge H_m(\partial N',p) + mKs.
\]
Since $Z\subset N^*$ and since $Z$ is an $(m,h)$ set, the left side of this inequality is at most $h$
(by the barrier principle~\ref{barrier-theorem}), so
\begin{align*}
   h &\ge H_m(\partial N',p) + mKs
   \\
   &\ge H_m(\partial N,p) - m\eps + mKs
\end{align*}
by~\eqref{H_ms-close}.  Since  $\eps>0$ can be arbitrarily small, this implies
that $h\ge H_m(\partial N,p) + mKs$ or
\[
  H_m(\partial N,p) \le h - mKs,
\]
from which it follows (Theorem~\ref{converse-theorem}) that $Z(s)$ is an $(m,h-mKs)$ set.

If $\dim(\Omega)=m+1$, then (letting $N$, $N'$, and $N^*$ be as above)
\[
  H_m(\partial N^*,q) \ge H_m(\partial N',p) + \rho s
\]
by Proposition~\ref{eigenvalue-proposition}.
Arguing exactly as above with $\rho$ in place of $mK$, we conclude
that
\[  
   H_m(\partial N, p) \le h - \rho s,
\]
from which it follow that $Z(s)$ is an $(m, h-\rho s)$ set.
\end{proof}

\section{Appendix: Tubular Neighborhoods}

For the reader's convenience, we give  the basic facts about
the second fundamental form of the level sets of the distance function
to a smooth hypersurface.  (These facts were used in Section~\ref{distance-section}.)

\begin{lemma}\label{tubular-lemma}
Let $M$ be a two-sided, smoothly embedded hypersurface in 
a smooth, $(n+1)$-dimensional Riemannian manifold $N$, 
let $f:N\to\RR$ be the signed distance function to $M$, 
and let $\Omega$ be an open subset of $N$ on which $f$ is smooth
with nonvanishing gradient.
For $p\in \Omega$, let 
\[
   M_p := \{x: f(x)=f(p)\}
\]
be the level set of $f$ containing $p$, and let $B_p$ be the second
fundamental form of $M_p$ at $p$ with respect to the unit normal $\nu(p):=\nabla f(p)$.
Then
\[
  (\D_\nu B) (\cdot,\cdot) = \Rr(\cdot,\nu,\cdot,\nu) + \sum_{k=1}^n B(\cdot,\ee_k)B(\ee_k,\cdot).
\]
where $R$ is the curvature tensor of $N$ and where 
 $\ee_1,\dots,\ee_n$ are unit vectors orthogonal to each other and to $\nu$.
\end{lemma}

\begin{proof}  
Note that the hypotheses imply that $\nu$ is a unit vectorfield and that  the integral
curves of $\nu$ are geodesics:
\begin{equation}\label{nu-derivative-zero}
  \D_\nu \nu \equiv 0.
\end{equation}
Let $\xx$ and $\yy$ be two tangent vectorfields to one of the level sets of $f$.
Extend these vectorfields by parallel transport along the integral curves of $\nu$.
Thus
\begin{equation}\label{x-and-y-derivative-zero}
  \D_\nu\xx = \D_\nu\yy = 0.
\end{equation}
Now
\[
  B(\yy,\xx)=B(\xx,\yy) = (\D_\xx\yy)\cdot \nu = -\xx\cdot \D_{\yy}\nu.
\]
By~\eqref{x-and-y-derivative-zero}, $(\D_\nu B)(\xx,\yy) = \nu(B(\xx,\yy))$.  
Thus by~\eqref{nu-derivative-zero} and~\eqref{x-and-y-derivative-zero}, 
\begin{align*}
(\D_\nu B)(\xx,\yy)
&=
\nu(\D_\xx\yy\cdot\nu)
\\
&=
(\D_\nu\D_\xx\yy)\cdot \nu + \D_\xx\yy \cdot\D_\nu\nu
\\
&=
(\D_\nu\D_\xx\yy)\cdot \nu + 0  
\\
&=
\Rr(\nu,\xx)\yy \cdot \nu + (\D_\xx\D_\nu\yy)\cdot \nu + (\D_{[\nu,\xx]}\yy)\cdot \nu
\\
&=
\Rr(\nu,\xx,\nu,\yy) + 0 + (\D_{[\nu,\xx]}\yy)\cdot \nu
\end{align*}
It remains only to show that
\begin{equation}\label{to-show}
   (\D_{[\nu,\xx]}\yy)\cdot \nu = \sum_{k=1}^n B(\xx,\ee_k) B(\ee_k,\yy).
\end{equation}
Now
\[
  [\nu,\xx] = \D_{\nu}\xx - \D_{\xx}\nu = - \D_{\xx}\nu
\]
which is orthogonal to $\nu$, and thus tangent to the level sets of $f$, so
\begin{equation}\label{Lie-bracket-term}
 (\D_{[\nu,\xx]}\yy)\cdot \nu = B([\nu,\xx], \yy) =  B(-\D_{\xx}\nu, \yy).
\end{equation}
Now
\begin{align*}
 -\D_{\xx}\nu
 &=
 \sum_{k=1}^n (-\D_{\xx}\nu\cdot \ee_k) \ee_k
 \\
 &=
 \sum_{k=1}^n B(\xx,\ee_k) \ee_k
\end{align*}
Substituting this into~\eqref{Lie-bracket-term} and using the linearity of $B(\cdot,\yy)$ gives~\eqref{to-show}.
\end{proof}

\begin{proposition}\label{eigenvalue-proposition}
Let $\Omega$, $f$, $\nu$, $B$, and $M_x$ (for $x\in\Omega$) be as in Lemma~\ref{tubular-lemma}.
Let $\kappa_1(x)\le \dots \le \kappa_n(x)$ be the principal curvatures of $M_x$
at $x$ with respect to the unit normal $\nu$.
Let $\Gamma$ be a geodesic curve perpendicular to the level sets of $f$ (i.e., an integral
curve of the vectorfield $\nu:=\nabla f$).

If $p,q\in \Gamma$ and if $f(q)>f(p)$, then
\begin{align}
\label{principal-curvature-inequality}
\kappa_i(q) &\ge \kappa_i(p) + K\dist(p,q), 
\\
\label{trm-inequality}
\trm B_q &\ge \trm B_p + mK \dist(p,q), 
\\
\label{mean-curvature-inequality}
H(q) &\ge H(p) + \rho \dist(p,q)
\end{align}
where $K$ is a lower bound for the sectional curvature of the ambient space, $\rho$ is a lower
bound for the Ricci curvature of the ambient space, and $H(x)=\operatorname{trace}B_x$
is the mean curvature of $M(x)$ at $x$ with respect to $\nu$.
\end{proposition}

\begin{proof}
Let $V$ be the space of normal vectorfields $\vv$ on $\Gamma$ such that $\D_\nu\vv\equiv 0$.
We may regard $B(q)$ and $B(p)$ as both being symmetric bilinear forms on $V$.
(In effect, we are identifying $\Tan_pM_p$ and $\Tan_qM_q$ by parallel transport along $\Gamma$.)

Let $\vv$ be a vectorfield in $V$.
Then $\D_\nu (B(\vv,\vv))= (\D_\nu B)(\vv,\vv)$, so by Lemma~\ref{tubular-lemma},
\begin{equation}\label{B-inequality}
\D_\nu (B(\vv,\vv)) 
\ge
\Rr(\nu,\vv,\nu,\vv).
\end{equation}
and thus
\[
\D_\nu (B(\vv,\vv)) \ge K \|\vv\|^2.
\]
Integrating from $p$ to $q$ gives
\[
   B_q(\vv,\vv) \ge B_p(\vv,\vv) + K \dist(p,q) \|\vv\|^2
\]
Now~\eqref{principal-curvature-inequality}
 follows from the Rayleigh quotient characterization of the eigenvalues of $B_x$, i.e., the principal
curvatures.  (See Lemma~\ref{rayleigh-lemma} below.)

Summing from $i=1$ to $m$ in~\eqref{principal-curvature-inequality} gives~\eqref{trm-inequality}.

To prove~\eqref{mean-curvature-inequality}, let $\ee_1,\dots, \ee_n$ 
be an orthonormal set of vectorfields in $V$.  Then by~\eqref{B-inequality},
\[
  \D_\nu (B(\ee_i,\ee_i)) \ge \Rr(\nu,\ee_i,\nu,\ee_i).
\]
Summing from $i=1$ to $n$ gives
\[
   \D_\nu h \ge \operatorname{Ricci}(\nu,\nu) \ge \rho.
\]
Assertion~\eqref{mean-curvature-inequality} follows by integrating from $p$ to $q$.
\end{proof}

\begin{lemma}\label{rayleigh-lemma}
Let $Q$ and $Q'$ be symmetric bilinear forms on a Euclidean space $V$.
Suppose $Q(\vv,\vv)\le Q'(\vv,\vv)$ for all unit vectors $\vv$.  Then each eigenvalue
of $Q$ is less than or equal to the corresponding eigenvalue of $Q'$.
\end{lemma}

\begin{proof} This follows immediately from the Rayleigh quotient characterization of the
eigenvalues:
\[
   \lambda_k(Q) = \inf_{W\in G(k,V)\vphantom{\hat\hat N}} \left( \sup_{w\in W,\, |w|=1} Q(w,w) \right),
\]
where $G(k,V)$ is the set of $k$-dimensional linear subspaces of $V$.
\end{proof}

\newcommand{\hide}[1]{}

\end{document}